\documentclass[reqno,10pt]{amsart}
\usepackage{amssymb,amsmath,amsthm,amsfonts,color}

\usepackage{diffcoeff}

\usepackage{mathrsfs}

\usepackage{ulem}
\usepackage{enumerate}
\usepackage{threeparttable}

\usepackage{bm}
\usepackage[left=2.8cm,right=2.8cm,top=2.8cm,bottom=2.8cm]{geometry}
\usepackage{mathtools}
\usepackage{graphicx}
\usepackage[format=hang,labelfont=bf]{caption}

\usepackage[colorlinks=true, pdfstartview=FitV, linkcolor=blue,
            citecolor=blue, urlcolor=blue]{hyperref}

\allowdisplaybreaks
\numberwithin{equation}{section}
\theoremstyle{plain}
\newtheorem{theorem}{Theorem}[section]
\newtheorem{proposition}[theorem]{Proposition}

\theoremstyle{definition}
\newtheorem{definition}[theorem]{Definition}
\newtheorem{remark}[theorem]{Remark}

\newtheorem*{theorem*}{Theorem}

\newcommand{\calk}{\mathcal K}
\newcommand{\frkm}{\mathfrak m}
\DeclareMathOperator{\diag}{diag}

\begin{document}
\title{Dissipativeness of the hyperbolic quadrature method of moments for kinetic equations}

\author{Ruixi Zhang}
\address{Department of Mathematical Sciences, Tsinghua University\\
    Beijing, 100084, China}
\email{1553548358@qq.com}

\author{Yihong Chen}
\address{Department of Mathematical Sciences, Tsinghua University\\
    Beijing, 100084, China}
\email{chenyiho20@mails.tsinghua.edu.cn}

\author{Qian Huang*}
\address{Institute of Applied Analysis and Numerical Simulation, University of Stuttgart\\
    Stuttgart, 70569, Germany}
\thanks{* Corresponding author}
\email{qian.huang@mathematik.uni-stuttgart.de; hqqh91@qq.com}

\author{Wen-An Yong}
\address{Department of Mathematical Sciences, Tsinghua University\\
    Beijing, 100084, China \\
    Beijing Institute of Mathematical Sciences and Applications\\
    Beijing 101408, China}
\email{wayong@tsinghua.edu.cn}


\vskip .2truecm
\begin{abstract}
  This paper presents a dissipativeness analysis of a quadrature method of moments (called HyQMOM) for the one-dimensional BGK equation. The method has exhibited its good performance in numerous applications. However, its mathematical foundation has not been clarified. Here we present an analytical proof of the strict hyperbolicity of the HyQMOM-induced moment closure systems by introducing a polynomial-based closure technique. As a byproduct, a class of numerical schemes for the HyQMOM system is shown to be realizability preserving under CFL-type conditions. We also show that the system preserves the dissipative properties of the kinetic equation by verifying a certain structural stability condition. The proof uses a newly introduced affine invariance and the homogeneity of the HyQMOM and heavily relies on the theory of orthogonal polynomials associated with realizable moments, in particular, the moments of the standard normal distribution.
\end{abstract}

\maketitle

\normalem

\section{Introduction}

We are interested in a class of moment closure methods dealing with the hypothetical 1D BGK equation for the velocity distribution $f=f(t,x,\xi)$ with $t >0$, $x \in \mathbb R$ and $\xi \in \mathbb R$:
\begin{equation} \label{eq:1D-BGK}
  \left \{
  \begin{aligned}
    \partial_t f+\xi \partial_x f &= \frac{1}{\tau}(f^{eq}-f), \\
    f^{eq} = f^{eq}(\xi; \rho,U,\theta) &=
    \rho \phi_{\sqrt{\theta}} (\xi-U).
  \end{aligned}
  \right.
\end{equation}
Here $\tau$ is a relaxation time, the density $\rho$, mean velocity $U$ and temperature $\theta$ are defined as
\[
  \rho = \int_{\mathbb R} fd\xi,\quad
  \rho U = \int_{\mathbb R} \xi fd\xi,\quad
  \rho\theta + \rho U^2 = \int_{\mathbb R} \xi^2fd\xi,
\]
and
\begin{equation}
  \phi_{\sigma}(\xi) = \frac{1}{\sqrt{2\pi}\sigma} \exp \left( -\frac{\xi^2}{2\sigma^2} \right).
\end{equation}
In this way, the local Maxwellian $f^{eq}$ is completely specified.

The kinetic equation (\ref{eq:1D-BGK}) is a simplification of the Boltzmann equation governing spatial-velocity distributions of molecules with binary collisions \cite{Cercignani1988}. The latter is well accepted as the underlying theory of rarefied gas dynamics and has wide applications. Examples include spacecraft reentry and lunar-lander-induced dusty plumes where flows transition from continuum to rarefied and free molecular regimes \cite{LZH2019,Morris2011}. In the dilute Martian atmosphere, the micro-sized dust suspensions have a Knudsen number of $\mathcal O(1)$, indicating as well a rarefied flow \cite{ClementsMars2013}. Generally, the kinetic model is also a powerful tool to study systems of interacting agents, like particulate flows \cite{HuangEF2} and active matter dynamics \cite{March2013}.

But it is a challenging task to solve the Boltzmann equation due to high dimensions and complicated collisions. In (\ref{eq:1D-BGK}), the collision term has been replaced by a simplified relaxation process \cite{bgk1954}. This is a widely used approximation for it preserves key properties of the original Boltzmann equation. We limit our analysis to the BGK equation. To tackle the high dimension, the method of moments stands out as an efficient strategy to reduce velocity dependence and generate `hydrodynamic' equations with macroscopic variables. In this process, a moment truncation to a finite order requires a closure of higher-order moments.

In pursuit of well-behaved moment closure methods, the 1D system (\ref{eq:1D-BGK}) is often considered as a starting point and test ground, where the $k$th moment is defined as
\begin{equation}
  M_k = M_k(t,x) = \int_{\mathbb R} \xi^k f d\xi
\end{equation}
for $k\in\mathbb{N}$.
Then the governing equation for $M_k$ can be derived from (\ref{eq:1D-BGK}) as
\begin{equation} \label{eq:unc}
  \partial_t M_k + \partial_x M_{k+1}= \frac{1}{\tau} \left(\rho \Delta_k(U,\theta)-M_k \right),
\end{equation}
with
\begin{equation} \label{eq:standard_mom}
  \Delta_k(u,\sigma^2) := \int_{\mathbb R} \xi^k \phi_\sigma(\xi-u)d\xi.
\end{equation}
Clearly, a truncation with $\mathbf M = (M_0,\dots,M_N)$ needs a closure $M_{N+1} = M_{N+1}(\mathbf M)$ for the hierarchy (\ref{eq:unc}) to be closed.

In this paper, we are concerned with quadrature-based methods of moments to close (\ref{eq:unc}). Unlike the well-known Grad's method \cite{Grad1949}, the quadrature method (abbreviated as QMOM) relies on nonlinear reconstructions of distributions that can be far from equilibrium and, by its very nature, preserves positivity. Precisely, $2n$ moments ($N=2n-1$ in (\ref{eq:unc})) are invoked to construct a convex combination of $n$ Dirac $\delta$-functions with centers and weights being unknowns \cite{Mc1997}. Theoretically, this is permissible for all realizable moments. The centers are found as roots of an orthogonal polynomial induced by the moments. In this way, the QMOM is numerically efficient and becomes popular in simulating particulate flows \cite{HuangEF2,MarFox2013}.
However, the QMOM-induced system is proved to be \textit{non}-hyperbolic and can give unphysical shocks in simulation \cite{Chalons2012,Fox2008,Huang2020}.

To remedy this drawback, many efforts have been paid to develop hyperbolic moment closure systems. One strategy called extended-QMOM (EQMOM) is to introduce a new (unknown) parameter and approximate the distribution by a convex combination of $n$ parameter-dependent approximations of the $\delta$-function using $2n+1$ moments ($N=2n$ in (\ref{eq:unc})) \cite{Chalons2010}. If the approximation is Gaussian, hyperbolicity is proved in \cite{Chalons2017,Huang2020}. For more general approximations, we refer to \cite{ZHY2023}.
But the domain of the EQMOM-induced systems (namely, the moment set admitting such a reconstruction) does not contain all realizable moments \cite{Chalons2017}.

Another strategy called hyperbolic-QMOM (HyQMOM) can be free of these drawbacks. The point is to take $N=2n$ in (\ref{eq:unc}) and find a convex combination of more than $n$ $\delta$-functions (with unknown centers and weights) to recover $2n+1$ moments $\mathbf M=(M_0,\dots,M_{2n})$ with more freedom. For this underdetermined problem, it is possible to pick a proper reconstruction that closes $M_{2n+1}(\mathbf M)$ and yields a hyperbolic system. This can be done in many ways \cite{Boh2020,Fox2018}.
Recently, general HyQMOM closures up to any order are proposed in \cite{FoxLau2022} and \cite{van2021} with the aid of the orthogonal polynomial theory associated with realizable moments \cite{Gau2004}. Applications of the HyQMOM to rarefied flows, multiphase flows and plasma simulations exhibit advantages over other approaches in viability, accuracy and computational efficiency \cite{FoxLau2022,LiMar2022,Mueller2023}.
Based on these, it was conjectured in \cite{FoxLau2022,van2021} that the HyQMOM closure systems are hyperbolic.
However, an analytic proof of the hyperbolicity is still lacking, although it was shown by computing the corresponding characteristic polynomials only for the closure in \cite{FoxLau2022} with $n\le 9$ using MATLAB symbolic.

One of our objectives here is to present a purely analytic proof of the hyperbolicity of the HyQMOM for both closures in \cite{FoxLau2022,van2021} with any integer $n$. The proof uses a factorization of the characteristic polynomial for the resultant first-order PDE and a polynomial-based closure technique.
As a byproduct, a class of numerical schemes for the HyQMOM system is shown to be realizability preserving under CFL-type conditions.

Besides hyperbolicity, the dissipativeness of moment systems should also be addressed to ensure that the dissipative property of the original kinetic equation are correctly inherited. For this purpose, we will also show that the HyQMOM-induced moment system satisfies the structural stability condition proposed in \cite{Yong1999}, which is believed to be a proper counterpart of the $H$-theorem for the kinetic equation. Indeed, the stability condition has been shown to be respected by many classical physical models \cite{Yong2008} and moment closure systems \cite{Di2017,Huang2022,Zhao2017}. A violation of the condition may lead to exponentially-exploding asymptotic solutions \cite{LiuJW2016}. For the quadrature methods, the condition has been verified for the EQMOM with Gaussian kernels \cite{Huang2020}, but there do exist kernels with which the EQMOM system contradicts the stability condition \cite{ZHY2023}.

The proof for the dissipativeness involves seeking positive solutions to an overdetermined system of algebraic equations. It uses a newly introduced affine invariance and homogeneity of the HyQMOM and heavily relies on the theory of orthogonal polynomials associated with realizable moments, in particular, the moments of the standard normal distribution.

The remainder of the paper is organized as follows. Section \ref{sec:qbmm} introduces the theory of orthogonal polynomials associated with realizable moments, several quadrature-based methods of moment, including QMOM, EQMOM and HyQMOM, and the structural stability condition. Section \ref{sec:result} contains our main results. In section \ref{sec:hyp}, we prove the strict hyperbolicity of the HyQMOM-induced moment system. Section \ref{sec:stab} is devoted to verifying the structural stability condition. Finally, we conclude our paper in section \ref{sec:concld}.

\section{Quadrature-based method of moments}
\label{sec:qbmm}

Given an integer $N\geq3$, a direct truncation of (\ref{eq:unc}) for $\mathbf M = (M_0,\dots,M_N)$ yields a system
\begin{equation} \label{eq:unc1}
  \partial_t \mathbf M + \partial_x \mathcal F(\mathbf M) = S(\mathbf M),
\end{equation}
with
\[
\begin{aligned}
  \mathcal F(\mathbf M) &= (M_1,\dots,M_N,M_{N+1}(\mathbf M))\in\mathbb R^{N+1}, \\
  S(\mathbf M) &= \frac{1}{\tau} (0,0,0,\rho \Delta_3(U,\theta)-M_3,\dots, \rho \Delta_N(U,\theta)-M_N)\in\mathbb R^{N+1}.
\end{aligned}
\]
Once $M_{N+1}(\mathbf M)$ is determined (either explicitly or implicitly), the first-order system of PDEs (\ref{eq:unc1}) becomes a moment closure system in the conservative form.

\subsection{Orthogonal polynomials} \label{subsec:orthpoly}
A vector (moment) $\mathbf M = (M_0,\dots,M_{2n})$ of odd dimension is called \textit{strictly realizable} if the corresponding Hankel matrix
\begin{equation}
  H_n = H_n(\mathbf M) =
  \begin{bmatrix}
    M_0 & M_1 & \cdots & M_n \\
    M_1 & M_2 & \cdots & M_{n+1} \\
    \vdots & \vdots & & \vdots \\
    M_n & M_{n+1} & \cdots & M_{2n}
  \end{bmatrix}
  \in \mathbb R^{(n+1)\times(n+1)}
\end{equation}
is positive definite \cite{gtm277}. This means that $\mathbf M$ is generated by a physically relevant velocity distribution.
Denote by
\[
  \Omega_{2n} = \{\mathbf M\in\mathbb R^{2n+1} | H_n(\mathbf M) \text{ is positive definite} \}
\]
the collection of strictly realizable moments. Clearly, $\Omega_{2n}$ is a positive cone.

For $\mathbf M \in\Omega_{2n}$, a linear functional $\langle \cdot \rangle_{\mathbf M}$ can be defined on $\mathbb R[X]_{2n}$ (the space of real
polynomials of degree $\le 2n$) as
\begin{equation} \label{eq:functl}
  \langle X^k \rangle_{\mathbf M} \mapsto M_k
\end{equation}
for $k=0,1,\dots,2n$.
The notation $\langle \cdot \rangle$ is used for simplicity if the dependence on $\mathbf M$ is clear.
The functional induces an inner product on $\mathbb R[X]_{n}$ as
\begin{equation}
  (p,q) \mapsto \langle pq \rangle
\end{equation}
for $p,q\in \mathbb R[X]_n$.
To see this, just notice that if $p=\sum_{k=0}^n p_k X^k$, then we have $\langle p^2 \rangle = \bm p^T H_n \bm p \ge 0$ with $\bm p=(p_0,\dots,p_n)^T\in\mathbb R^{n+1}$.

Given the inner product, a family of monic orthogonal polynomials $Q_k = Q_k(X)$ ($k=0,1,\dots,n$) can be constructed and satisfy
\[
  \deg Q_k = k \quad \text{and} \quad \langle Q_i Q_k \rangle = \langle Q_k^2 \rangle \delta_{ik}
\]
for $0\le i,k \le n$. Here $\delta_{ik}$ denotes the Kronecker delta. The orthogonal polynomials are generated recursively as $Q_{-1}=0$, $Q_0=1$ and
\begin{equation} \label{eq:qrec}
  Q_{k+1} = (X - a_k)Q_k - b_kQ_{k-1}
\end{equation}
for $k=0,1,\dots,n-1$. The coefficients $a_k$ and $b_k$ are derived as \cite{Gau2004}
\begin{equation} \label{eq:ab}
  a_k = \frac{\langle XQ_k^2 \rangle}{\langle Q_k^2 \rangle}; \quad
  b_0 = M_0, \quad
  b_k = \frac{\langle Q_k^2 \rangle}{\langle Q_{k-1}^2 \rangle}>0\ (k\geq1).
\end{equation}
From these expressions it is not difficult to see that $a_k$ depends on $M_0,\dots,M_{2k+1}$ (with linear dependence on $M_{2k+1}$) and $b_k$ depends only on $M_0,\dots,M_{2k}$ (with linear dependence on $M_{2k}$). Notice that $b_{n}$ can be well defined although it does not appear in (\ref{eq:qrec}).

About the orthogonal polynomials, we have the following important facts.

\begin{proposition}[Theorem 1.20 of \cite{Gau2004}] \label{prop:inter}
  Each $Q_k$ has $k$ distinct real roots which are separated by those of $Q_{k-1}$.
\end{proposition}

\begin{proposition}[Theorem 5.10 of \cite{gtm277}] \label{prop:bij1}
  The relation (\ref{eq:ab}), together with (\ref{eq:qrec}), defines a bijection between $\Omega_{2n}$ and
  \[
    \{(a_0,\dots,a_{n-1},b_0,\dots,b_n)\in \mathbb R^{2n+1} | b_i>0 \text{ for } i=0,\dots,n\}.
  \]
\end{proposition}

On the other hand, a vector (moment) $\mathbf M' = (M_0,\dots,M_{2n-1})$ of even dimension is called strict realizable if it belongs to
\[
  \Omega_{2n-1} = \{\mathbf M'\in\mathbb R^{2n} | H_{n-1}(\mathbf M') \text{ is positive definite} \}.
\]
For $M'\in\Omega_{2n-1}$, we can also define a functional $\langle \cdot \rangle_{\mathbf M'}$ on $\mathbb R[X]_{2n-1}$ as in (\ref{eq:functl}) and correponding orthogonal polynomials $Q_0,\dots,Q_{n-1}\in\mathbb R[X]_{n-1}$. Notice that now $a_{n-1}$ and $Q_n$ are also well defined. In this case, we have the following analogue of Proposition \ref{prop:bij1}:

\begin{proposition}[Theorem 5.10 of \cite{gtm277}] \label{prop:bij2}
  The relation (\ref{eq:ab}), together with (\ref{eq:qrec}), defines a bijection between $\Omega_{2n-1}$ and
  \[
    \{(a_0,\dots,a_{n-1},b_0,\dots,b_{n-1})\in \mathbb R^{2n} | b_i>0 \text{ for } i=0,\dots,n-1\}.
  \]
\end{proposition}

Furthermore, we have the following important fact.
\begin{proposition}[Theorem 9.7 of \cite{gtm277}] \label{prop:qmomrecons}
  The map
  \[
    \{ W=(w_i,u_i)_{i=1}^n \in\mathbb R^{2n} | w_i>0,\ \forall i, \ u_1<u_2<\dots<u_n \} \longrightarrow
    \Omega_{2n-1}
  \]
  with
  \[
    M_j = \sum_{i=1}^n w_i u_i^j, \quad j=0,\dots,2n-1
  \]
  is a bijection. Moreover, $u_1,\dots,u_n$ are just the zeros of $Q_n$.
\end{proposition}

\subsection{Quadrature methods of moments} \label{subsec:qmom}
Here we introduce two widely used quadrature methods of moments. The first one is the quadrature method of moments (QMOM).
In this method, we take $N=2n-1$ in (\ref{eq:unc1}) and close (\ref{eq:unc1}) by specifying $M_{2n}=M_{2n}(\mathbf M')$ in the following way:
\begin{equation} \label{eq:qm2}
  M_{2n} = \sum_{i=1}^n w_i u_i^{2n},
\end{equation}
where $w_i$ and $u_i$ are uniquely determined by $\mathbf M'\in\Omega_{2n-1}$ according to Proposition \ref{prop:qmomrecons}. Equivalently, the QMOM can be understood as the ansatz $f(\xi)=\sum_{i=1}^nw_i\delta(\xi-u_i)$.

This closure is equivalent to $\langle Q_n^2 \rangle_\mathbf M  = 0$ with the functional induced by $\mathbf M=(\mathbf M',M_{2n})$. Indeed, from definition (\ref{eq:functl}) we can easily see that
\begin{equation} \label{eq:pMbrac}
  \langle p \rangle_{\mathbf M'} = \sum_{i=1}^n w_i p(u_i),\quad p\in\mathbb R[X]_{2n-1}.
\end{equation}
Then we deduce from Proposition \ref{prop:qmomrecons} that
\[\begin{aligned}
\langle Q_n^2\rangle_{\mathbf M}=&\langle X^{2n}\rangle_{\mathbf M}-\langle X^{2n} - Q_n^2 \rangle_{\mathbf M}=\langle X^{2n}\rangle_{\mathbf M}-\langle X^{2n} - Q_n^2 \rangle_{\mathbf M'} \\
=& M_{2n}-\sum_i w_i (u_i^{2n} - Q_n^2(u_i))=M_{2n}-\sum_i w_i u_i^{2n}.
\end{aligned}\]
Consequently, $\langle Q_n^2 \rangle_\mathbf M  = 0$ is equivalent to (\ref{eq:qm2}).

While the QMOM becomes a popular method in solving size distributions of the particulate flow \cite{HuangEF2,MarFox2013}, it is found that the QMOM-induced system (\ref{eq:unc1}) is not hyperbolic \cite{Chalons2012,Huang2020}. Therefore, various improvements have been proposed in literature.
One of them is the extended QMOM (EQMOM).

The EQMOM assumes that the distribution $f(\xi)$ is a weighted summation of $n$ shifted homoscedastic kernels:
\[
  f(\xi) = \sum_{i=1}^n w_i \delta_{\sigma}(\xi;\xi_i).
\]
This Ansatz contains $2n+1$ unknowns collected as $W=(w_i>0,\xi_i\in\mathbb R,\sigma>0)_{i=1}^n\in\mathbb R^{2n+1}$.
Taking moments of the Ansatz up to $2n$-th order leads to a map from $W$ to $\mathbf M$.
If this map is invertible, the unclosed $M_{2n+1}$ in (\ref{eq:unc1}) can be computed in terms of $W$ and thereby the moments $\mathbf M$.
If the kernel $\delta_\sigma (\xi;\xi_i)$ is chosen to be Gaussian $\phi_\sigma(\xi-\xi_i)$ \cite{Chalons2010}, the resultant moment closure system (\ref{eq:unc1}) has been shown in \cite{Huang2020} to be strict hyperbolic and satisfy the structural stability condition. More general kernels
\begin{equation} \label{eq:eqker}
  \delta_\sigma (\xi;\xi_i) = \frac{1}{\sigma} \calk \left( \frac{\xi-\xi_i}{\sigma} \right)
\end{equation}
are considered in \cite{ZHY2023}, and a class of $\calk(\xi)$ is identified to ensure the hyperbolicity and structural stability.

A possible drawback of EQMOM is that the image of the map, or equivalently, the domain of (\ref{eq:unc1}) after closure, is a proper subset of $\Omega_{2n}$ (see Proposition 3.1 in \cite{Chalons2017} for the Gaussian case with $n=2$). Therefore, it cannot handle all realizable moments.

\subsection{Hyperbolic QMOM} \label{subsec:hypqm}
The focus of this paper is the following improvement called hyperbolic QMOM (HyQMOM) \cite{FoxLau2022,van2021}. In this method, we take $N=2n$ in (\ref{eq:unc1}) and seek a closure for $M_{2n+1}=M_{2n+1}(\mathbf M)$ on the realizable domain $\Omega_{2n}$.

For any $M_{2n+1}\in\mathbb R$ and $\mathbf M\in\Omega_{2n}$, we write $\tilde{\mathbf M} = (\mathbf M, M_{2n+1}) \in \mathbb R^{2n+2}$ and define a functional $\langle \cdot \rangle_{\tilde{\mathbf M}}$ on $\mathbb R[X]_{2n+1}$ as in subsection \ref{subsec:orthpoly}.
Consequently, we get the orthogonal polynomials $Q_0,\dots,Q_{n+1}\in\mathbb R[X]_{n+1}$ and real numbers $a_0,\dots,a_n,b_0,\dots,b_n$.
Since the restriction of $\langle \cdot \rangle_{\tilde{\mathbf M}}$ on $\mathbb R[X]_{2n}$ is just $\langle \cdot \rangle_{\mathbf M}$, we know from Proposition \ref{prop:bij1} that $a_0,\dots,a_{n-1},b_0,\dots,b_n$ are determined by $\mathbf M$. Therefore, there exists a one-to-one correspondence between $a_n$ and $M_{2n+1}$.

Thanks to the above reasoning, the HyQMOM closure for $M_{2n+1}$ can be achieved by taking \cite{FoxLau2022}
\begin{equation} \label{eq:hyqmom}
  a_n = \frac{1}{n} \sum_{k=0}^{n-1}a_k.
\end{equation}
A generalized version of HyQMOM reads as \cite{van2021}
\begin{equation} \label{eq:ghyq}
  a_n=\frac{\gamma}{n} \sum_{k=0}^{n-1}a_k
\end{equation}
with $\gamma>-2n$ a scaling factor. Since $a_n$ is linearly dependent on $M_{2n+1}$ (see (\ref{eq:ab})), these relations do give convenient closures for $M_{2n+1}$.

While the HyQMOM exhibits good performance in numerical simulations \cite{FoxLau2022,LiMar2022,Mueller2023}, its theoretical understanding is far from satisfactory. The hyperbolicity is verified by resorting to the MATLAB only for the closure (\ref{eq:hyqmom}) with $n\le 9$, and is conjectured for the general closure (\ref{eq:ghyq}) \cite{FoxLau2022,van2021}. Moreover, the dissipativeness of HyQMOM-induced moment systems remains unclear.

\subsection{Structural stability condition} \label{subsec:ssc}

Given a closure $M_{N+1}(\mathbf M)$, the moment system (\ref{eq:unc1}) of first-order PDEs has a coefficient matrix
\[
  A(\mathbf M):= \diffp{{\mathcal F(\mathbf M)}}{{\mathbf M}} \in \mathbb R^{(N+1)\times (N+1)}
\]
on the domain $\mathbb G$.
The system (\ref{eq:unc1}) is called \textit{hyperbolic} if $A(\mathbf M)$ has $(N+1)$ linearly independent real eigenvectors \cite{Serre1999}. If $A(\mathbf M)$ has $(N+1)$ distinct real eigenvalues, it is called \textit{strictly hyperbolic}.
Moreover, the structural stability condition proposed in \cite{Yong1999} is a weak version of certain entropy conditions which characterize the dissipativeness of the moment system. Herein we restate it for the 1D system (\ref{eq:unc1}) with a non-empty equilibrium manifold $\mathcal E = \{\mathbf M \in \mathbb G \mid S(\mathbf M)=0\}$.
Denote by $S_\mathbf M(\mathbf M)$ the Jacobian matrix of $S(\mathbf M)$. The structural stability condition reads as
\begin{itemize}
    \item [(I)] For any $\mathbf M \in \mathcal E$, there exist invertible matrices $P=P(\mathbf M)\in \mathbb R^{(N+1) \times (N+1)}$ and $\hat T = \hat T(\mathbf M) \in \mathbb R^{r \times r}$ ($0<r \le N+1$) such that
    \[
    P S_\mathbf M(\mathbf M) P^{-1} = \text{diag}(\textbf{0}_{(N+1-r) \times (N+1-r)}, \hat T).
    \]

    \item [(II)] For any $\mathbf M \in \mathbb G$, there exists a positive definite symmetric matrix $A_0 = A_0(\mathbf M)$ such that $A_0 A(\mathbf M) = A^T(\mathbf M) A_0$.

    \item [(III)] For any $\mathbf M \in \mathcal E$, the coefficient matrix and the source are coupled as
    \[
        A_0 S_\mathbf M (\mathbf M) + S_\mathbf M^T (\mathbf M) A_0 \le - P^T
        \begin{bmatrix}
            0 & 0 \\
            0 & I_r
        \end{bmatrix}
        P.
    \]
\end{itemize}
Here $I_r$ is the unit matrix of order $r$.

\begin{remark} \label{rem:sym}
  For the 1-D system (\ref{eq:unc1}), Condition (II) is satisfied if and only if the system is hyperbolic \cite{Huang2020}.
  Condition (III) is a proper manifestation of the dissipative property inherited from the kinetic model. See detailed discussions in \cite{Yong2001}.

  Recently, it has been demonstrated that several moment models respect the structural stability condition \cite{Di2017,Huang2022,Zhao2017}. For the quadrature-based methods, a series of analyses was performed in our previous works for the EQMOM with Gaussian and more general kernels (\ref{eq:eqker}) \cite{Huang2020,ZHY2023}.
\end{remark}

\section{Main results} \label{sec:result}

In this section, we present our main results. The first one is

\begin{theorem} [Hyperbolicity] \label{thm:hyp}
  The HyQMOM closure (\ref{eq:ghyq}) yields a strictly hyperbolic moment system (\ref{eq:unc1}) for $\mathbf M \in \Omega_{2n}$.
\end{theorem}

This theorem means that the coefficient matrix $A(\mathbf M)$ has $2n+1$ distinct eigenvalues $\lambda_0<\dots<\lambda_{2n}$ for $\mathbf M \in \Omega_{2n}$. An interesting consequence of this theorem is
\begin{proposition} \label{prop:hqrecons2}
  For the moment system (\ref{eq:unc1}) with the closure (\ref{eq:ghyq}) for $\gamma>-n$, there exist $\{\omega_i>0\}_{i=0}^{2n}$ such that
  \[
    M_k = \sum_{i=0}^{2n} \omega_i \lambda_i^k, \quad k=0,\dots,2n.
  \]
\end{proposition}

As for the structural stability of HyQMOM, we have

\begin{theorem} [Structural stability] \label{thm:stab}
  The moment system (\ref{eq:unc1}) with the HyQMOM closure (\ref{eq:hyqmom}) satisfies the structural stability condition for $\mathbf M \in \Omega_{2n}$.
\end{theorem}

Let us remark that Theorem \ref{thm:stab} is proved only for the closure (\ref{eq:hyqmom}). The proof relies on the following property of the HyQMOM closure (\ref{eq:hyqmom}). To state this property, we refer to \cite{FoxLau2022} and introduce
\begin{definition}
A moment closure
\[
  \mathcal M: \mathbf M=(M_0,\dots,M_N)\in\mathbb R^{N+1} \mapsto (\mathbf M,M_{N+1})\in\mathbb R^{N+2}
\]
is referred to as affine invariant if it satisfies
\begin{equation} \label{eq:affinv}
  \mathcal M \circ \mathcal S_N^{[u,\sigma]} = \mathcal S_{N+1}^{[u,\sigma]} \circ \mathcal M
\end{equation}
for any $u\in\mathbb R, \ \sigma>0$ and a family of parametrized linear operators defined as
\begin{equation} \label{eq:Sk}
  \mathcal S^{[u,\sigma]}_k: (M_0,\dots,M_k)\in\mathbb R^{k+1} \mapsto
  (\bar{M}_0,\dots,\bar{M}_k)\in\mathbb R^{k+1}
\end{equation}
with
\[
   \bar{M}_k = \sum_{j=0}^k \binom{k}{j} \sigma^j M_j u^{k-j}.
\]
\end{definition}

The affine invariance means that if a moment closure is achieved through a distribution reconstruction $f(\xi;\mathbf M)$ parametrized with moments $\mathbf M$, then the closure for $M_{N+1}$ is invariant after the distribution is shifted and rescaled as
\begin{equation} \label{eq:scalef}
  f(\xi;\mathcal S_N^{[u,\sigma]}(\mathbf M)) = \frac{1}{\sigma} f\left(\frac{\xi-u}{\sigma}; \mathbf M\right)
\end{equation}
in the sense of moments up to order $N+1$. In view of this, we believe that the affine invariance is a natural requirement for moment closure methods.

For the HyQMOM closure (\ref{eq:ghyq}), we have
\begin{proposition} \label{prop:hyqafinv}
  The generalized HyQMOM (\ref{eq:ghyq}) is affine invariant if and only if $\gamma = 1$, which is just the closure (\ref{eq:hyqmom}).
\end{proposition}
This property of the closure (\ref{eq:hyqmom}) is crucial for the proof of structural stability, as is seen in section \ref{sec:stab}. In addition, we have (with a proof in Appendix)
\begin{proposition} \label{prop:afinvold}
  Both the QMOM closure and EQMOM closure with kernels (\ref{eq:eqker}) are affine invariant.
\end{proposition}

As an interesting application of the reconstruction based on Proposition \ref{prop:qmomrecons}:
\[
f(\xi;\mathbf M)=\sum_{i=0}^{n}w_i\delta(\xi-u_i), \qquad\mathbf M\in\Omega_{2n}
\]
or the reconstruction in Proposition \ref{prop:hqrecons2},
we analyze the realizability of the first-order upwind scheme \cite{Fox2018} for the moment system (\ref{eq:unc1}):
\begin{equation} \label{eq:upwind}
  \mathbf M^{p+1}_j = \mathbf M^p_j + \frac{\Delta t}{\Delta x_j} \left( \mathcal F^p_{j-\frac{1}{2}} - \mathcal F^p_{j+\frac{1}{2}} \right) +
  \frac{\Delta t}{\tau} \left(\rho \bm \Delta^p_j - \mathbf M_j^{p+1} \right).
\end{equation}
Here $\bm \Delta = (\Delta_0(U,\theta), \dots, \Delta_{2n}(U,\theta))$ and $\mathcal F_{j\pm\frac{1}{2}} = (\mathcal F_0,\dots,\mathcal F_{2n})_{j\pm\frac{1}{2}}$ is the spatial flux.
The subscript $j$ represents the $j$th cell $]x_{j-\frac{1}{2}}, x_{j+\frac{1}{2}}[$ of size $\Delta x_j$, and the superscript $p$ denotes the time $t_p = p\Delta t$.
With the reconstructed distribution, the `kinetic-based' flux can be taken as
\begin{equation} \label{eq:kbflux1}
  \mathcal F_{k,j+\frac{1}{2}} = \sum_{i=0}^{n} \left. w_i \max \{ 0,u_i \}^{k+1} \right |_j + \sum_{i=0}^{n} \left. w_i \min\{ 0,u_i \}^{k+1} \right |_{j+1}
\end{equation}or
\begin{equation} \label{eq:kbflux2}
  \mathcal F_{k,j+\frac{1}{2}} = \sum_{i=0}^{2n} \left. \omega_i \max \{ 0,\lambda_i \}^{k+1} \right |_j + \sum_{i=0}^{2n} \left. \omega_i \min\{ 0,\lambda_i \}^{k+1} \right |_{j+1}.
\end{equation}
For these two schemes, we have

\begin{theorem} [Realizability preserving] \label{thm:bp}
  The scheme (\ref{eq:upwind}) with flux (\ref{eq:kbflux1}) or (\ref{eq:kbflux2}) is realizable under the CFL condition
  \[
    \frac{\Delta t}{\Delta x_j} \max_i |u_{i,j}| \le 1 \quad
    \text{or} \quad
    \frac{\Delta t}{\Delta x_j} \max_i |\lambda_{i,j}| \le 1,
    \quad \forall j.
  \]
\end{theorem}

\begin{proof}
We only consider the flux (\ref{eq:kbflux1}) because the same arguments work for (\ref{eq:kbflux2}).
  A straightforward calculation gives
  \[
  \begin{aligned}
    \mathcal F_{k,j-\frac{1}{2}} - \mathcal F_{k,j+\frac{1}{2}} =& \sum_{i=0}^{n} \left. w_i \max\{0,u_i\}^{k+1} \right|_{j-1} - \sum_{i=0}^{n} \left. w_i \min\{ 0,u_i \}^{k+1} \right |_{j+1} \\
    &+ \sum_{i=0}^{n} \left. w_i \left[ \min\{ 0,u_i \}^{k+1} - \max \{ 0,u_i \}^{k+1} \right] \right|_j \\
    =& \sum_{i=0}^{n} w_i |u_i| \max\{ 0,u_i \}^k \Big|_{j-1}
    + \sum_{i=0}^{n} w_i |u_i| \min\{ 0,u_i \}^k \Big|_{j+1}\\
&- \sum_{i=0}^{n} w_i |u_i| u_i^k \Big|_j.
  \end{aligned}
  \]
  Notice that the first two terms, denoted by $T$, are in $\Omega_{2n}$. Then we see from (\ref{eq:upwind}) that
  \[
    \left(1+\frac{\Delta t}{\tau}\right) M_{k,j}^{p+1} = \frac{\Delta t}{\Delta x_j} T^p + \sum_{i=0}^{n} w_i \left( 1-\frac{\Delta t}{\Delta x_j} |u_i| \right) u_i^k \Big|_j^p
    +\frac{\Delta t}{\tau}\rho\Delta_k(U,\theta) \big|_j^p
  \]
  for $k=0,\dots,2n$.
  Clearly, the positivity of $( 1-\Delta t |u_i| / \Delta x_j )$ ensures $\mathbf M_j^{p+1} \in \Omega_{2n}$ (the positive cone). This completes the proof.
\end{proof}

The following sections are devoted to the proofs of the other conclusions above. Precisely, section \ref{sec:hyp} is for Theorem \ref{thm:hyp} and Proposition \ref{prop:hqrecons2}, section \ref{sec:stab} for Theorem \ref{thm:stab} and Proposition \ref{prop:hyqafinv} on the affine invariance, and an appendix for Proposition \ref{prop:afinvold}.

\section{Hyperbolicity} \label{sec:hyp}

We prove Theorem \ref{thm:hyp} and Proposition \ref{prop:hqrecons2} in this section. The key is the following proposition.

\begin{proposition} \label{prop:hypf}
  For the HyQMOM closure (\ref{eq:ghyq}), the characteristic polynomial of the corresponding coefficient matrix $A(M)$ can be factorized as
  \[
    F(X;M)=Q_n \cdot \left[ (X-a_n)Q_n - \frac{2n+\gamma}{n} b_n Q_{n-1} \right]:=Q_nR_{n+1}.
  \]
\end{proposition}

Let us mention that this proposition was proved in \cite{FoxLau2022} only for the closure (\ref{eq:hyqmom}) with $n\le 9$ by the symbolic software. Here we will present a purely analytical proof at the end of this section.

Having this proposition, we can prove Theorem \ref{thm:hyp} and Proposition \ref{prop:hqrecons2} as follows.
First of all, it is not difficult to deduce from Proposition \ref{prop:inter} that $R_{n+1}$ has $n+1$ distinct real roots which are separated by those of $Q_n$ (see the proof of Theorem 4.8 in \cite{FoxLau2022}, where the arguments work for $\gamma>-2n$). Thus, the characteristic polynomial $F(X;\mathbf M)$ has $2n+1$ distinct real roots and the strict hyperbolicity follows. This proves Theorem \ref{thm:hyp}.

For Proposition \ref{prop:hqrecons2}, we have
\begin{proof}[Proof of Proposition \ref{prop:hqrecons2}]
  As a solution to the linear system of algebraic equations with Vandermone coefficients, the uniqueness and existence of such $\omega_i$'s are obvious. Thus, we only need to show the positivity.
  For this purpose, we write $\tilde{\mathbf M} = (\mathbf M,M_{2n+1})$, $\mathbf M=(\mathbf M',M_{2n})$ and $\mathbf M'=(M_0,\dots,M_{2n-1})$.
  It is clear that the restriction of $\langle \cdot \rangle_{\tilde{\mathbf M}}$ on $\mathbb R[X]_{2n}$ (resp. $\mathbb R[X]_{2n-1}$) is just $\langle \cdot \rangle_{\mathbf M}$ (resp. $\langle \cdot \rangle_{\mathbf M'}$).

  Then we recall that $R_{n+1}$ has $n+1$ distinct roots separated by those of $Q_n$. Denote by $\{\lambda_{2i}\}_{i=0}^n$ the roots of $R_{n+1}$ and by $\{\lambda_{2i-1}\}_{i=1}^n$ the roots of $Q_n$. Thus, it follows from Proposition \ref{prop:qmomrecons} that there exist $\{w_i'>0\}_{i=1}^{n}$ such that
  \begin{equation} \label{eq:wip}
    \sum_{i=1}^n w_i' \lambda_{2i-1}^k = \left\{
    \begin{aligned}
      &M_k,             \quad &&k=0,\dots,2n-1, \\
      &M_{2n} - \langle Q_n^2 \rangle_\mathbf M, \quad &&k=2n.
    \end{aligned}
    \right.
  \end{equation}
  Here the equality for $k=2n$ is due to the relation (\ref{eq:pMbrac}) leading to
  \[
    0 = \sum_{i=1}^n w_i' Q_n^2(\lambda_{2i-1})= \sum_{i=1}^n w_i' \lambda_{2i-1}^{2n} + \langle Q_n^2 - X^{2n} \rangle_{\mathbf M'}.
  \]

  On the other hand, it is not difficult to verify that $Q_0,\dots,Q_n,R_{n+1}=(X-a_n)Q_n - \frac{2n+\gamma}{n} b_n Q_{n-1}$ are orthogonal polynomials based on the inner product induced by the realizable moments $\tilde{\mathbf M}'' = (\mathbf M', M_{2n}'', M_{2n+1}'')$ with $M_{2n}''=M_{2n} + \frac{n+\gamma}{n} \langle Q_n^2 \rangle_\mathbf M$. Again by Proposition \ref{prop:qmomrecons}, there exist $\{w_i''>0\}_{i=0}^{n}$ such that
  \begin{equation} \label{eq:wipp}
    \sum_{i=0}^n w_i'' \lambda_{2i}^k = \left\{
    \begin{aligned}
      &M_k,             \quad &&k=0,\dots,2n-1, \\
      &M_{2n} + \frac{n+\gamma}{n} \langle Q_n^2 \rangle_\mathbf M, \quad &&k=2n.
    \end{aligned}
    \right.
  \end{equation}

  With (\ref{eq:wip}) and (\ref{eq:wipp}), it is straightforward to see that
  \[
    \sum_{i=1}^n \frac{n+\gamma}{2n+\gamma} w_i' \lambda_{2i-1}^k + \sum_{i=0}^n \frac{n}{2n+\gamma} w_i'' \lambda_{2i}^k = M_k
  \]
  for $k=0,\dots,2n$. This together with the uniqueness of $\{\omega_i\}_{i=0}^{2n}$ clearly implies that
  \[
    \omega_{2i-1} = \frac{n+\gamma}{2n+\gamma} w_i' \quad \text{and} \quad
    \omega_{2i} = \frac{n}{2n+\gamma} w_i''
  \]
  are both positive if $n+\gamma>0$. This completes the proof.
\end{proof}

Now we turn to the proof of Proposition \ref{prop:hypf}.
Instead of computing the characteristic polynomial as in \cite{FoxLau2022}, we develop a polynomial-based closure method.
Given a monic polynomial with coefficients depending on the known moment vector $\mathbf M$:
\[
  G=G(X;\mathbf M)=\sum_{k=0}^N g_k(\mathbf M)X^k + X^{N+1},
\]
we can define a moment closure as
\begin{equation} \label{eq:clos}
  M_{N+1} = \langle X^{N+1}-G \rangle_{\mathbf M} = - \sum_{k=0}^N g_k(\mathbf M) M_k.
\end{equation}
From this perspective, the HyQMOM can be written as
\[
M_{2n+1}=\langle X^{2n+1}-(X-a_n)Q_n^2\rangle_{\mathbf M}
\]
with $a_n=a_n(\mathbf M)=\frac{\gamma}{n}\sum_{k=0}^{n-1}a_k(\mathbf M)$ given in (\ref{eq:ghyq})
and the QMOM can be written as $M_{2n}=\langle X^{2n}-Q_n^2\rangle_{\mathbf M'}$.

For the resultant moment closure, we have

\begin{proposition} \label{prop:clos}
The charateristic polynomial of the closure system induced by (\ref{eq:clos}) is the given $G(X;\mathbf M)$ if and only if
  \[
    \langle \partial_\mathbf M G \rangle_{\mathbf M} = \bm 0.
  \]
\end{proposition}

\begin{proof}
First we recall that for the general moment system (\ref{eq:unc1}), the characteristic polynomial of the coefficient matrix $A(\mathbf M)$ (of degree $N+1$) can be written as
\begin{equation} \label{eq:chp}
  F=F(X;\mathbf M) = \sum_{j=0}^\infty c_j X^j
\end{equation}
with $c_j(\mathbf M)=-\partial M_{N+1}/\partial M_j$ for $j=0,\dots,N$, $c_{N+1}=1$ and $c_j=0$ for $j\ge N+2$.
  Using this and (\ref{eq:clos}), we compute
  \[
  \begin{aligned}
    F &= -\sum_{k=0}^N \frac{\partial M_{N+1}}{\partial M_k} X^k + X^{N+1}
    = \sum_{k=0}^N \left( g_k + \sum_{j=0}^N \frac{\partial g_j}{\partial M_k} M_j \right) X^k + X^{N+1} \\
    &= G + \sum_{k,j=0}^N \frac{\partial g_j}{\partial M_k} M_j X^k.
  \end{aligned}
  \]
  It is clear that $F=G$ if and only if $\sum_{j=0}^N (\partial_{M_k} g_j) M_j = 0$ for $k=0,\dots,N$.
\end{proof}

With this proposition, we can choose proper polynomial $G$ and construct new closure systems that are strictly hyperbolic. One simple example is

\begin{proposition} \label{prop:newclos}
  Take $N=2n-1$ and $G(X;\mathbf M')=Q_n^2-Q_{n-1}^2$, where the orthogonal polynimials are generated by $\mathbf M'\in \Omega_{2n-1}$. The moment closure system induced by (\ref{eq:clos})
  \[
    M_{2n} = \langle X^{2n}-Q_n^2 +Q_{n-1}^2  \rangle_{\mathbf M'}
  \]
  is strictly hyperbolic.
\end{proposition}

\begin{proof}
  First, we refer to Proposition \ref{prop:clos} and compute
  \[
    \langle \partial_{M_i} G \rangle_{\mathbf M'} = 2 \langle Q_n \partial_{M_i} Q_n \rangle_{\mathbf M'} - 2 \langle Q_{n-1} \partial_{M_i} Q_{n-1} \rangle_{\mathbf M'} = 0,
  \]
  where the last step is due to the orthogonality of $Q_k$ and $\partial_{M_i}Q_k$ (a polynomial of degree $\le k-1$). This means that the characteristic polynomial is just $G(X;\mathbf M')$.
  Then, we show that $G(X;\mathbf M')$ has $2n$ distinct real roots. Denote by $\{\lambda_{2i-1} \}_{i=1}^n$ and $\{\lambda_{2i} \}_{i=1}^{n-1}$ the distinct roots of $Q_n$ and $Q_{n-1}$, respectively.
  Because of Proposition \ref{prop:inter}, we may assume $\lambda_1<\lambda_2<\dots<\lambda_{2n-1}$. This implies $\text{sgn} G(\lambda_i;\mathbf M') = (-1)^i$ for $i=1,\dots,2n-1$ and hence $G(X;\mathbf M')$ has $2n-2$ distinct real roots. The additional two roots are due to $G(\lambda_1;\mathbf M')<0$, $G(\lambda_{2n-1};\mathbf M')<0$ and $G(X;\mathbf M')\to\infty$ as $|X|\to\infty$.
\end{proof}

Furthermore, we can prove the following proposition in the appendix.
\begin{proposition} \label{prop:qeqhomo}
  The QMOM and EQMOM with kernels (\ref{eq:eqker}) are just the closure (\ref{eq:clos}) by choosing $G(X;\mathbf M)$ as the corresponding charateristic polynomials.
\end{proposition}

We conclude this section with a proof of Proposition 4.1.
\begin{proof}[Proof of Proposition \ref{prop:hypf}]
  Set $G=Q_nR_{n+1}$ with $R_{n+1}=(X-a_n)Q_n-c_nQ_{n-1}$ and $c_n=\frac{2n+\gamma}{n}b_n$. Recall that $\langle XQ_n^2\rangle_{\mathbf{\tilde M}}=a_n\langle Q_n^2\rangle_\mathbf M$ with $a_n=a_n(\mathbf M)$ given in (\ref{eq:ghyq}) and thereby
  \[
  \begin{aligned}
    M_{2n+1} &=\langle X^{2n+1}-XQ_n^2\rangle_{\mathbf M}+a_n\langle Q_n^2\rangle_{\mathbf M}\\
    &= \langle X^{2n+1}-(X-a_n)Q_n^2)\rangle_\mathbf M\\
    &= \langle X^{2n+1}-G\rangle_\mathbf M.
  \end{aligned}
  \]
  Here we have used the orthogonality $\langle Q_nQ_{n-1}\rangle_\mathbf M=0$.

  Next we show that $G\in\mathbb R[X]_{2n+1}$ is just the charateristic polynomial corresponding to the closure (\ref{eq:ghyq}). By Proposition 4.2, it suffices to show $\langle\partial_\mathbf M G\rangle_\mathbf M=0$. To do this, we firstly write $Q_n=X^n-\sigma X^{n-1}+[\cdots]$ with $[\cdots]$ a polynomial in $\mathbb R[X]_{n-2}$ and deduce from (\ref{eq:qrec}) that
  \[
  \begin{aligned}
    Q_n &= (X-a_{n-1})Q_{n-1}-b_{n-1}Q_{n-2}\\
    &= (X-a_{n-1})((X-a_{n-2})Q_{n-2}-b_{n-2}Q_{n-3})-b_{n-1}Q_{n-2}\\
    &= \cdots\\
    &= (X-a_{n-1})(X-a_{n-2})\cdots(X-a_0)Q_0+[\cdots]\\
    &= X^n-\sum_{k=0}^{n-1}a_kX^{n-1}+[\cdots].
  \end{aligned}
  \]
  This implies $\sigma=\sum_{k=0}^{n-1}a_k$ and thereby
  \[
    a_n  =\frac{\gamma}{n}\sum_{k=0}^{n-1}a_k= \frac{\gamma}{n}\sigma.
  \]
  With this relation, we use the orthogonality and compute
  \[
  \begin{aligned}
    \langle \partial_{\mathbf M} G\rangle=&\langle ((X-a_n)Q_n-c_nQ_{n-1})\partial_{\mathbf M}Q_n\rangle+\langle Q_n\partial_{\mathbf M}((X-a_n)Q_n-c_nQ_{n-1})\rangle\\
    =&\langle (X-a_n)Q_n\partial_{\mathbf M}Q_n\rangle-c_n\langle Q_{n-1}\partial_{\mathbf M}Q_n\rangle+\langle (X-a_n)Q_n\partial_{\mathbf M}Q_n\rangle\\
    &+\langle Q_n^2\partial_{\mathbf M} (X-a_n)\rangle-c_n\langle Q_n\partial_{\mathbf M}Q_{n-1}\rangle-\langle Q_nQ_{n-1}\partial_{\mathbf M}c_n\rangle\\
    =&-2\langle Q_n^2\rangle\partial_{\mathbf M}\sigma-c_n\langle Q_{n-1}^2\rangle\partial_{\mathbf M}\sigma-\langle Q_n^2\rangle\partial_{\mathbf M}a_n\\
    =& \left(-2+\frac{c_n}{b_n}-\frac{\gamma}{n} \right)\langle Q_n^2\rangle\partial_{\mathbf M}\sigma=0.
  \end{aligned}
  \]
This completes the proof.
\end{proof}

\section{Structural stability} \label{sec:stab}

This section is devoted to a proof of Theorem \ref{thm:stab}. Namely, we will show that the HyQMOM induced moment system satisfies the structural stability condition (I)-(III) in subsection \ref{subsec:ssc}. Because Condition (II) is equivalent to the  hyperbolicity, we only need to verify Conditions (I) and (III). Moreover, the both conditions need to be examined only on the equilibrium manifold
\[
 \mathcal E = \{ (\rho \Delta_0(U,\theta), \dots, \rho \Delta_{N}(U,\theta)) \in\mathbb R^{N+1} \ | \ \rho>0, U\in\mathbb R, \theta>0 \}
\]
with
\[
 \rho=M_0, \quad
 U = \frac{M_1}{M_0}, \quad
 \theta = \frac{M_0M_2-M_1^2}{M_0^2}.
\]
Clearly, these three equalities define a bijection between $(\rho, U, \theta)$ and $(M_0,M_1, M_2)$.

We first consider Condition (I). The Jacobian of the source $S(\mathbf M)$ in (\ref{eq:unc1}) reads as
\[
  S_\mathbf M =
  \begin{bmatrix}
    \bm 0_{3\times 3} & \\
    \hat S_\mathbf M & -I_{N-2}
  \end{bmatrix},
\]
where the $(N-2)\times 3$ matrix $\hat S_\mathbf M$ is
\[
  \hat S_\mathbf M = \diffp{(\rho \Delta_3(U,\theta),\dots,\rho \Delta_{N}(U,\theta))}{{(M_0,M_1,M_2)}}.
\]
It is straightforward to show that for
\begin{equation} \label{eq:Pinv}
  P^{-1} =
  \begin{bmatrix}
    Y & \\
    \hat S_\mathbf M Y & I_{N-2}
  \end{bmatrix}
  \quad \text{with} \quad
  Y=\diffp{{(M_0,M_1,M_2)}}{{(\rho,U,\theta)}} \in \mathbb R^{3\times 3},
\end{equation}
we have
\[
  S_\mathbf M P^{-1} = P^{-1}
  \begin{bmatrix}
    \bm 0_{3\times 3} & \\
    & -I_{N-2}
  \end{bmatrix},
\]
thus justifying Condition (I) for (\ref{eq:unc1}). Clearly, Condition (I) holds for any moment closure method leading to a system of first-order PDEs in form of (\ref{eq:unc1}).

\subsection{A sketch for verifying Condition (III)}
\label{subsec:cond3}

The verification of Condition (III) consists of several steps.
The first step is to find a symmetrizer $A_0=A_0(\mathbf M)$ of the coefficient matrix $A(\mathbf M)$ with distinct eigenvalues $\lambda_0<\dots<\lambda_{N}$.
As shown in \cite{Huang2020}, the symmetrizer $A_0$ must be in the form $L^TDL$ with
\begin{equation} \label{eq:L}
\begin{split}
  L &=
  \begin{bmatrix}
    \lambda_0^{N} & \lambda_0^{N-1} & \cdots & \lambda_0 & 1 \\
    \lambda_1^{N} & \lambda_1^{N-1} & \cdots & \lambda_1 & 1 \\
    \lambda_2^{N} & \lambda_2^{N-1} & \cdots & \lambda_2 & 1 \\
    \vdots & \vdots & & \vdots & \vdots \\
    \lambda_{N}^{N} & \lambda_{N}^{N-1} & \cdots & \lambda_{N} & 1
  \end{bmatrix}
  \begin{bmatrix}
    1 &&&& \\
    c_{N} & 1 &&& \\
    c_{N-1} & c_{N} & 1 && \\
    \vdots & \vdots & \vdots & \ddots & \\
    c_1 & c_2 & c_3 & \cdots & 1
  \end{bmatrix} \\[4mm]
  &= (F_k(\lambda_i))_{0\le i,k\le N}
\end{split}
\end{equation}
and
\[
  D=\diag(\omega_0,\dots,\omega_{N}) \in \mathbb R^{(N+1)\times (N+1)}
\]
a positive definite matrix \textit{to be determined}.
Here $F_k(X)$ stands for
\begin{equation} \label{eq:Fk}
  F_k(X) = \sum_{j=k+1}^\infty c_j X^{j-k-1}
\end{equation}
for $k=0,\dots,N$ as polynomials of degree $N-k$ and $F_k(X)=0$ for $k\ge N+1$.

By Theorem 2.1 in \cite{Yong1999}, Condition (III) is satisfied if and only if there exists a symmetrizer $A_0$ such that the matrix
\begin{equation} \label{eq:KM}
  K(\mathbf M) := P^{-T} A_0 P^{-1} = (LP^{-1})^T D (LP^{-1})
\end{equation}
is of the block-diagonal form $\diag \left(K'_{3\times 3},\ K''_{(N-2)\times(N-2)} \right)$ for equilibrium states $\mathbf M\in\mathcal E$.
With $L$ given in (\ref{eq:L}) and $P^{-1}$ given in (\ref{eq:Pinv}), our task is to demonstrate the existence of a positive definite diagonal matrix $D=\diag(\omega_i)_{i=0}^{N}$.

Before preceeding, we remark that, for $\mathbf M\in\mathcal E$, the characteristic polynomial $F(X;\mathbf M)$ and the coefficients $c_j(\mathbf M)$ in (\ref{eq:chp}) are only dependent on $(U, \theta)$.
  To see this, we deduce from Proposition \ref{prop:hypf} and (\ref{eq:qrec}) that $F(X;\mathbf M)$ and $c_j(\mathbf M)$ are determined by the coefficients $a_k(\mathbf M)$ and $b_k(\mathbf M)$ in (\ref{eq:ab}), where the involved bracket $\langle \cdot \rangle$ is defined as $\langle X^k \rangle \mapsto \rho\Delta_k(U,\theta)$ for $\mathbf M\in\mathcal E$.
  Therefore, it is clear from (\ref{eq:ab}) that $a_k(\mathbf M)$ and $b_k(\mathbf M)$ are independent of $\rho$ at equilibrium states, so too are $F(X;\mathbf M)$ and $c_j(\mathbf M)$.
Thanks to this fact, we write $F(X;\mathbf M)$ and $c_j(\mathbf M)$ as $F(X;U,\theta)$ and $c_j(U,\theta)$, respectively.
  Furthermore, the eigenvalues $\lambda_i(\mathbf M)$ ($i=0,\dots,N$) (roots of $F(X;\mathbf M)$) and the polynomials $F_k(X;\mathbf M)$ in (\ref{eq:Fk}) are all dependent on $(U,\theta)$ for $\mathbf M\in\mathcal E$, denoted $\lambda_i(U,\theta)$ and $F_k(X;U,\theta)$, respectively.

Now we calculate $K(\mathbf M)$ defined in (\ref{eq:KM}). Notice that the first three columns of $P^{-1}$ in (\ref{eq:Pinv}) are
\begin{equation} \label{eq:pinvcol}
  (\Delta_j(U,\theta), \ \rho \partial_U \Delta_j(U,\theta), \ \frac{\rho}{2} \partial_U^2 \Delta_j(U,\theta))_{j=0}^{N} \in \mathbb R^{(N+1)\times 3},
\end{equation}
where the expression of the third column is due to $(\partial_\theta - \frac{1}{2}\partial_U^2)\Delta_j(U,\theta)=0$ (see, for example, Lemma 4.1 (c) and (d) in \cite{Huang2020}).
Then we use (\ref{eq:L}) and (\ref{eq:pinvcol}) to obtain the $i$th row of the matrix $LP^{-1}$ as
\begin{equation} \label{eq:lpinvrow}
  \left( h_0(\lambda_i;U,\theta), \ \rho h_1(\lambda_i;U,\theta), \ \frac{\rho}{2}h_2(\lambda_i;U,\theta), \ F_3(\lambda_i), \dots, F_{N}(\lambda_i) \right),
\end{equation}
where the polynomial $h_j(X;U,\theta)$ is defined as
\begin{equation} \label{eq:hj}
\begin{split}
  h_j(X;U,\theta) = \sum_{k=0}^{\infty} F_k(X) \partial_U^j \Delta_k(U,\theta)
  =\sum_{k=0}^\infty X^k \left[ \sum_{l=0}^\infty c_{l+k+1} \partial_U^j \Delta_l(U,\theta) \right].
\end{split}
\end{equation}
Here the dependence on $(U,\theta)$ is omitted for $F_k(X)$, $\lambda_i$ and $c_j$ for clarity. Recall from \cite{Huang2020} that $\Delta_l(U,\theta)$ is a polynomial of degree $l$ in $U$. Then it is not difficult to see that the summation over $k$ actually goes from 0 to $N-j$ in the second equality.

Note that $F_k(\lambda_i)$ is a linear combination of $\lambda_i^\beta$ ($0\le\beta\le N-k$). We see from (\ref{eq:lpinvrow}) that $K(\mathbf M) = (LP^{-1})^TD(LP^{-1})$ attains the block-diagonal form if and only if the $3(N-2)$ relations
\begin{equation} \label{eq:goal}
  \sum_{i=0}^{N} \omega_i h_j(\lambda_i;U,\theta) \lambda_i^\beta = 0
\end{equation}
hold for $j=0,1,2$ and $\beta=0,1,\dots,N-3$. Recall that $N\geq3$.
Consequently, we only need to prove
\begin{proposition} \label{prop:target}
  For any $\mathbf M \in \mathcal E$, there exist $\{\omega_i>0\}_{i=0}^{N}$ such that the $3(N-2)$ equations in (\ref{eq:goal}) hold for $j=0,1,2$ and $\beta=0,\dots, N-3$.
\end{proposition}

It is quite cumbersome to handle (\ref{eq:goal}) directly. Fortunately, the HyQMOM closure (\ref{eq:hyqmom}) is \textit{affine invariant}, as stated in Proposition \ref{prop:hyqafinv} and proved in subsection \ref{subsec:hyafinv}.
Thank to this affine invariance, the characteristic polynomial has the following elegant form.

\begin{proposition} \label{prop:Faffinv}
  If the moment closure is affine invariant, then for any $\mathbf M \in \mathcal E$, the characteristic polynomial (\ref{eq:chp}) satisfies
\begin{equation}\label{eq:fhat}
  F(X;U,\theta) = \sigma^{N+1} \hat F \left( \frac{X-U}{\sigma} \right)
\end{equation}
and the polynomials defined in (\ref{eq:hj}) can be written as
\begin{equation} \label{eq:haffinv}
  h_j(X;U,\theta) = \sigma^{N-j} \hat h_j \left( \frac{X-U}{\sigma} \right)
\end{equation}
for $j=0,1,2$.
Here $\sigma = \sqrt\theta$ and $\hat \varphi(\cdot):=\varphi(\cdot;0,1)$.
\end{proposition}

A proof of this proposition is given in subsection \ref{subsec:consaffinv}.
As a result, if the moment closure is affine invariant, the eigenvalues can be expressed as $\lambda_i(U,\theta)=\sigma \hat \lambda_i + U$ with $\hat \lambda_i:=\lambda_i(0,1)$.
Consequently, Proposition \ref{prop:target} holds if the relations
\begin{equation} \label{eq:goal2}
  \mathcal A^\beta_j := \sum_{i=0}^{N} \omega_i \hat h_j(\hat \lambda_i) \hat \lambda_i^\beta = 0
\end{equation}
are satisfied for $j=0,1,2$ and $\beta=0,\dots,N-3$.

Substituting the definition of $h_j=h_j(X;U,\theta)$ in (\ref{eq:hj}) into the expression of $\mathcal A^\beta_j$ in (\ref{eq:goal2}) and setting
\begin{equation}\label{eq:pk}
  p_k = \sum_{i=0}^{N} \omega_i \hat \lambda_i^k, \quad \text{for } k=0,1,\dots,
\end{equation}
we can rewrite (\ref{eq:goal2}) as
\begin{equation} \label{eq:calAdef}
  \mathcal A_j^\beta = \sum_{k\ge0,\ l\ge0} \hat c_{k+l+1} p_{k+\beta} \partial^j \Delta_l
  = \sum_{k\ge\beta,\ l\ge-\beta} p_k \hat c_{k+l+1} \partial^j \Delta_{l+\beta}.
\end{equation}
Here $\Delta_j:=\Delta_j(0,1)$ and $\partial^j\Delta_l:=\partial_U^j\Delta_l(U,1)|_{U=0}$.
A frequently-used property of $p_k$ is
\begin{equation} \label{eq:pkpro}
  \sum_{k=0}^\infty \hat c_k p_{k+\beta} = \sum_{i=0}^{N} \omega_i \hat F(\hat \lambda_i) \hat \lambda_i^\beta = 0 \quad
  \text{for} \quad \beta \ge 0.
\end{equation}

Next we show that the $3(N-2)$ relations $\mathcal A_j^\beta = 0$ in (\ref{eq:goal2}) can be further reduced to $N$ equations for the HyQMOM closure (\ref{eq:hyqmom}), which is a special case of the following result to be proved in subsection \ref{subsec:consaffinv}.

\begin{proposition} \label{prop:reduce}
  Suppose that the moment closure $M_{N+1}(\mathbf M)$ is a homogeneous function of order one and affine invariant, and satisfies $M_{N+1}(\Delta_0,\dots,\Delta_{N})=\Delta_{N+1}$. Then for $\mathbf M \in \mathcal E$, the $3(N-2)$ relations in (\ref{eq:goal2}) hold if and only if $\mathcal A_0^\beta = 0$ for $\beta=0,\dots,N-1$.
\end{proposition}

\begin{remark} \label{rem:hyas}
  The HyQMOM closure (\ref{eq:hyqmom}) satisfies all the assumptions of this proposition with $N=2n$. Indeed, it is not difficult to see from (\ref{eq:ab}) and (\ref{eq:hyqmom}) that the HyQMOM closure $M_{2n+1}(\mathbf M)$ is homogeneous of order one and the affine invariance is due to Proposition \ref{prop:hyqafinv}. To show $M_{2n+1}(\Delta_0,\dots,\Delta_{2n}) = \Delta_{2n+1}$, we notice that the orthogonal polynomials induced by the moments $(\Delta_m)_{m=0}^{2n}$ are the scaled Hermite polynomials $\hat Q_k$ which are generated recursively as $\hat Q_{-1}=0$, $\hat Q_0=1$ and
  \[
    \hat Q_{k+1} = X \hat Q_k - k \hat Q_{k-1}
  \]
  for $k=0,\dots,n-1$ \cite{Gau2004}. This means $a_k = 0$ for $k\le n-1$ and $b_k = k$ for $k \le n$ in (\ref{eq:qrec}). Thus, the HyQMOM closure (\ref{eq:hyqmom}) gives $a_n=0$ and the resultant orthogonal polynomial of degree $n+1$ is just $\hat Q_{n+1}$. By Proposition \ref{prop:bij2} (taking $N=2n+1$) we conclude $M_{2n+1}(\Delta_0,\dots,\Delta_{2n}) = \Delta_{2n+1}$.
\end{remark}

Thanks to Proposition \ref{prop:reduce} and Remark \ref{rem:hyas}, the task remains to show $\mathcal A_0^{\beta}=0$ for $\beta=0,...,2n-1$ by choosing proper $p_k$ which are determined by $\omega_i$ as shown in (\ref{eq:pk}). In subsection \ref{subsec:pk}, we shall prove the following proposition.
\begin{proposition}\label{prop:pkchoice}
$\mathcal A_0^{\beta}=0$ for $\beta=0,...,2n-1$ provided that
\begin{equation} \label{eq:pksol}
  p_k = \left \{
  \begin{aligned}
    &\Delta_k,             \quad &&k=0,\dots,2n-1, \\
    &\Delta_{2n} + (n-1)!, \quad &&k=2n.
  \end{aligned}
 \right.
\end{equation}
\end{proposition}
For $p_k$ specified in (\ref{eq:pksol}), we can show as in Proposition \ref{prop:hqrecons2} that there exist $\{ \omega_i>0 \}_{i=0}^{2n}$ such that
\begin{equation} \label{eq:wpos}
  \sum_{i=0}^{2n} \omega_i \hat \lambda_i^k = p_k,\quad k=0,\dots,2n.
\end{equation}
In this way, the $\omega_i$'s are uniquely determined and thereby other $p_k$ appeared in (\ref{eq:calAdef}) are also defined.

To show (\ref{eq:wpos}), we note that $\langle \hat{Q}_n^2\rangle=n!$ and repeat the proof of Proposition \ref{prop:hqrecons2} in section \ref{sec:hyp} to see that there exist $\{w_i'>0\}_{i=1}^{n}$ such that
\[
  \sum_{i=1}^n \omega_i' \hat\lambda_{2i-1}^k = \left \{
  \begin{aligned}
    &\Delta_k, \quad && k=0,\dots, 2n-1, \\
    &\Delta_{2n} - n!, \quad && k=2n,
  \end{aligned} \right.
\]
and there exist $\{ \omega_i''>0 \}_{i=0}^n$ such that
\[
  \sum_{i=1}^n \omega_i'' \hat\lambda_{2i}^k = \left \{
  \begin{aligned}
    &\Delta_k, \quad && k=0,\dots, 2n-1, \\
    &\Delta_{2n} + (n+1)(n-1)!, \quad && k=2n.
  \end{aligned} \right.
\]
With the last two relations and (\ref{eq:pksol}), it is straightforward to verify that
\[
  \sum_{i=1}^n \frac{n}{2n+1} \omega_i' \hat \lambda_{2i-1}^k + \sum_{i=1}^n \frac{n+1}{2n+1} \omega_i'' \hat \lambda_{2i}^k = p_k
\]
for $k=0,\dots,2n$, which clearly implies
\[
  \omega_{2i-1} = \frac{n}{2n+1}\omega_i'>0 \quad \text{and} \quad
  \omega_{2i} = \frac{n+1}{2n+1}\omega_i''>0,
\]
thanks to the uniqueness of $\{\omega_i\}_{i=0}^{2n}$.

This completes the verification of Proposition \ref{prop:target} and hence Condition (III) is verified. We close this subsection with the following remark.

\begin{remark}
When verifying Condition (III) for the EQMOM with Gaussian kernels in \cite{Huang2020}, the problem is also reduced to $\mathcal A_0^\beta = 0$ for $\beta=0,\dots,2n-1$. Unlike the HyQMOM where $p_k$ is specified, the strategy for the EQMOM is to directly set $\omega_i=1$ such that all $\mathcal A_0^\beta$'s are null. Both proofs rely heavily on the detailed property of the coefficient matrix $A(\mathbf M)$.
\end{remark}

\subsection{Affine invariance}
\label{subsec:hyafinv}

In this subsection, we prove Proposition \ref{prop:hyqafinv} on the affine invariance of the HyQMOM closure (\ref{eq:hyqmom}).

\begin{proof} [Proof of Proposition \ref{prop:hyqafinv}]
  The HyQMOM closure introduces orthogonal polynomials $Q_0,\dots,Q_{n+1}$. Denote by $\{u_i\}_{i=0}^n$ the roots of $Q_{n+1}$. Proposition \ref{prop:qmomrecons} states that $M_k=\sum_{i=0}^n w_iu_i^k$ for $k=0,\dots,2n+1$ with $w_i>0$.
  Thus, the map $(\bar{\mathbf M},\bar{M}_{2n+1})=(\bar M_k)_{k=0}^{2n+1}:=\mathcal S_{2n+1}^{[u,\sigma]}(\mathcal M(\mathbf M)) $ can be expressed as
  \[
    \bar M_k = \sum_{j=0}^k \binom{k}{j} \sigma^j \left( \sum_{i=0}^n w_i u_i^j \right) u^{k-j}
    = \sum_{i=0}^n w_i (\sigma u_i + u)^k,\quad k=0,...,2n+1.
  \]

  On the other hand, the moments $\bar{\mathbf M}$ induce an inner product with orthogonal polynomials $\bar Q_0,\dots,\bar Q_n$. We claim that
  \begin{equation} \label{eq:tiltqq}
    \bar Q_k(X) = \sigma^k Q_k \left( \frac{X-u}{\sigma} \right)
  \end{equation}
  for $k=0,\dots,n$. To see this, denote by $\hat Q_k$ the RHS of (\ref{eq:tiltqq}) and notice that $\mathbf M$ and $\bar{\mathbf M}$ are both generated by finitely-supported distributions. Then we have
  \[
    \langle \hat Q_k \hat Q_l \rangle_{\bar{\mathbf M}} = \sum_{i=0}^n w_i (\hat Q_k \hat Q_l)(\sigma u_i + u)
    = \sigma^{k+l} \sum_{i=0}^n w_i (Q_kQ_l)(u_i)
    = \sigma^{k+l} \langle Q_k Q_l \rangle_\mathbf M,
  \]
  which indicates that the monic polynomials $\hat Q_0,\dots,\hat Q_n$ are orthogonal with respect to $\bar{\mathbf M}$. Thus, we obtain $\bar Q_k = \hat Q_k$ for $k=0,\dots,n$.

  Now consider $\mathcal M(\bar{\mathbf M})$. Substituting the recursion (\ref{eq:qrec}) into (\ref{eq:tiltqq}), we derive
  \[
    \bar Q_{k+1} = (X-u-\sigma a_k(\mathbf M)) \bar Q_k - \sigma^2 b_k(\mathbf M) \bar Q_{k-1}
  \]
  for $k=0,\dots,n-1$. This is the recursion for $\{\bar Q_k\}$ with
  \[
    a_k(\bar{\mathbf M}) = \sigma a_k(\mathbf M) + u, \quad
    b_k(\bar{\mathbf M}) = \sigma^2 b_k(\mathbf M).
  \]
  Performing the HyQMOM closure (\ref{eq:ghyq}) for both $\mathbf M$ and $\bar{\mathbf M}$ yields
  \[
    a_n(\bar{\mathbf M}) = \sigma a_n(\mathbf M) + \gamma u.
  \]
  Thus, for $\bar{\mathbf M}$ the HyQMOM closure yields the highest-order orthogonal polynomial
  \begin{equation} \label{eq:tqnp1}
  \begin{split}
    \bar Q_{n+1} &= (X-a_n(\bar{\mathbf M})) \bar Q_n - b_n(\bar{\mathbf M})\bar Q_{n-1} \\
    &= \sigma^{n+1} Q_{n+1} \left( \frac{X-u}{\sigma} \right) + (1-\gamma) u\sigma^n Q_n \left( \frac{X-u}{\sigma} \right).
  \end{split}
\end{equation}

  The closure is affine invariant, namely, $\mathcal M(\bar{\mathbf M}) = \mathcal S_{2n+1}^{[u,\sigma]}((\mathbf M,M_{2n+1}))$, meaning that
\[
    M_{2n+1}(\bar{\mathbf M})=\bar{M}_{2n+1} = \sum_{i=0}^n w_i (\sigma u_i + u)^{2n+1},
  \]
  which is valid if and only if the roots of $\bar Q_{n+1}$ are just $(\sigma u_i + u)_{i=0}^{n}$. From (\ref{eq:tqnp1}) we see that this is the case if and only if $\gamma=1$, which completes the proof.
\end{proof}

\subsection{Consequences of the affine invariance} \label{subsec:consaffinv}
In this subsection, we prove Propositions \ref{prop:Faffinv} and \ref{prop:reduce} as consequences of the affine invariance.

\begin{proof}[Proof of Proposition \ref{prop:Faffinv}]
 For (\ref{eq:fhat}) it suffices to show that the coefficients of $X^j$ of the two polynomials therein are equal:
  \begin{equation} \label{eq:cjaff}
    c_j(U,\theta) = \sigma^{N+1} \sum_{k=0}^\infty \binom{k}{j} \frac{\hat c_k}{\sigma^k} (-U)^{k-j},
  \end{equation}
  where the summation over $k$ actually ranges from $j$ to $N+1$.
 Note that the coefficients $c_j(U,\theta)$ and $\hat c_j$ are
  \begin{equation} \label{eq:cjeq}
    c_j(U,\theta) = -\diffp{{M_{N+1}}}{{M_j}} \Big|_{\mathbf M=(\Delta_m(U,\theta))_{m=0}^N}, \quad
    \hat c_j = -\diffp{{M_{N+1}}}{{M_j}} \Big|_{\mathbf M=(\Delta_m)_{m=0}^N}
  \end{equation}
  for $j=0,\dots,N$, $c_{N+1}(U,\theta)=\hat c_{N+1} = 1$, and $c_j(U,\theta)=\hat c_j = 0$ for $j\ge N+2$.

  To show (\ref{eq:cjaff}), we deduce from (\ref{eq:Sk}) that
  \[
    \mathcal S_k^{[-\frac{U}{\sigma}, \frac{1}{\sigma}]} ((\Delta_0(U,\theta),\dots,\Delta_k(U,\theta))) = (\Delta_0,\dots,\Delta_k)
  \]
  as $\Delta_k(U,\theta)$ and $\Delta_k$ are the moments of $\phi_\sigma (\xi-U)$ and $\phi(\xi)$, respectively (see (\ref{eq:standard_mom})).
  Moreover, for $\mathcal S_k^{[-\frac{U}{\sigma}, \frac{1}{\sigma}]}: (M_0,\dots,M_k) \mapsto (M_0',\dots,M_k')$ we have
  \begin{equation} \label{eq:Sk2}
    M_k' = \sigma^{-k} \sum_{j=0}^k \binom{k}{j} M_j (-U)^{k-j}.
  \end{equation}

  With these preparations, we differentiate the last component of the composite function
  \[
    \mathcal M \circ \mathcal S_N^{[-\frac{U}{\sigma}, \frac{1}{\sigma}]} = \mathcal S_{N+1}^{[-\frac{U}{\sigma}, \frac{1}{\sigma}]} \circ \mathcal M
  \]
 with respect to $M_j$ and evaluate the derivative at $\mathbf M= (\Delta_k(U,\theta))_{k=0}^N$ to obtain
  \[
  \begin{aligned}
    \sum_{k=0}^N (-\hat c_k) \left( \diffp{{M_k'}}{{M_j}} \right)
    = \diffp{{M_{N+1}'}}{{M_j}} + \left( \diffp{{M_{N+1}'}}{{M_{N+1}}} \right)
    (-c_j(U,\theta)).
  \end{aligned}
  \]
  Here (\ref{eq:cjeq}) and the chain rule have been used. Substituting $\partial M_k' / \partial M_j = \sigma^{-k} \binom{k}{j} (-U)^{k-j}$ from (\ref{eq:Sk2}) into the above expression immediately yields (\ref{eq:cjaff}). Hence, (\ref{eq:fhat}) is verified.

As to (\ref{eq:haffinv}),
we first show that
  \begin{equation} \label{eq:pUDeltak}
    \partial^j \Delta_k = \sigma^{j-k} \sum_{l=0}^k \binom{k}{l} (-U)^{k-l} \partial_U^j \Delta_l(U,\theta).
  \end{equation}
  In fact, we know from \cite{Huang2020} that
  \begin{equation} \label{eq:Deltader}
    \partial^j_U \Delta_k(U,\theta) = \frac{k!}{(k-j)!} \Delta_{k-j}(U,\theta)
  \end{equation}
  for $0\le j \le k$. This particularly implies
  \[
  \begin{aligned}
    \partial^j \Delta_k &= \int_\mathbb R \left( \diffp[j]{}{\xi} \xi^k \right) \phi(\xi)d\xi
    = \sigma^{j-k} \int_\mathbb R \left( \diffp[j]{}{\xi} (\xi-U)^k \right) \phi_\sigma (\xi-U) d\xi \\
    &= \sigma^{j-k} \sum_{l=0}^k \binom{k}{l} (-U)^{k-l} \int_\mathbb R \left( \diffp[j]{}{\xi} \xi^l \right) \phi_\sigma(\xi-U)d\xi  \\
    &= \sigma^{j-k} \sum_{l=0}^k \binom{k}{l} (-U)^{k-l} \partial_U^j \Delta_l(U,\theta),
  \end{aligned}
  \]
  where the last equality is due to (\ref{eq:Deltader}).

 Next we deduce from (\ref{eq:cjaff}) that
\[
  \begin{aligned}
    h_j(X;U,\theta) &= \sum_{k=0}^\infty X^k \sum_{l,m=0}^\infty \binom{m}{l+k+1} \hat c_m \sigma^{N+1-m} (-U)^{m-l-k-1} \partial_U^j \Delta_l(U,\theta) \\
    &= \sum_{k,l=0}^\infty X^k \hat c_l \sigma^{N+1-l} \sum_{m=0}^\infty \binom{l}{m+k+1} (-U)^{l-m-k-1} \partial_U^j \Delta_m(U,\theta).
  \end{aligned}
  \]
On the other hand, the RHS of (\ref{eq:haffinv}) reads as
  \[
  \begin{aligned}
    \hat h_j\left( \frac{X-U}{\sigma} \right)
    &= \sum_{k=0}^\infty \left(\frac{X-U}{\sigma}\right)^k \sum_{l=0}^\infty \hat c_{l+k+1} \partial^j \Delta_l \\
    &= \sum_{k=0}^\infty X^k \sum_{s,l=0}^\infty \sigma^{-s} \binom{s}{k} (-U)^{s-k} \hat c_{l+s+1} \partial^j \Delta_l \\
    &= \sum_{k,l=0}^\infty X^k \hat c_l \sum_{s=0}^\infty \sigma^{s+1-l} \binom{l-s-1}{k}(-U)^{l-s-1-k} \partial^j \Delta_s.
  \end{aligned}
  \]
  Here the second equality follows from expanding $(X-U)^k$ with a new dummy index $s$ and then exchanging the dummy indices $k,s$. The third equality is derived by introducing new dummy indices $l \leftarrow l+s+1$ and $s \leftarrow l$. Note that the summation over $s$ vanishes when $l=0$.
By using (\ref{eq:pUDeltak}), the coefficient of $X^k \hat c_l$ in the last expression can be written as
  \[
  \begin{aligned}
    &\sum_{s=0}^\infty \sigma^{j+1-l} \binom{l-s-1}{k}(-U)^{l-s-1-k} \sum_{m=0}^\infty \binom{s}{m} (-U)^{s-m} \partial_U^j \Delta_m(U,\theta) \\
    =& \sigma^{j+1-l} \sum_{m=0}^\infty (-U)^{l-m-1-k} \partial_U^j\Delta_m(U,\theta) \sum_{s=0}^\infty \binom{s}{m}\binom{l-s-1}{k}.
  \end{aligned}
  \]

Comparing this and the coefficient of $X^k \hat c_l$ in $h_j(X;U,\theta)$ above, it suffices to show the following interesting and elementary identity
\[
    \binom{l}{m+k+1} = \sum_{s=0}^\infty \binom{s}{m} \binom{l-s-1}{k}
\]
for nonnegative integers $k,l$ and $m$. Both sides of this identity are null if $m+k+1>l$ and then $s\ge m$ on the right-hand side leads to $l-s-1\le l-m-1<k$.
  For $m+k+1 \le l$, we consider the set of multiple indices
  \[
    A = \left\{\alpha=(\alpha_i\in\mathbb N)_{i=1}^{k+m+1}|1\le\alpha_1<\cdots<\alpha_{k+m+1}\le l \right\}.
  \]
  The number $|A|$ of elements in the set $A$ can be computed in two different ways, leading to
  \[
  \begin{aligned}
    \binom{l}{m+k+1} = |A| &= \sum_{s=0}^\infty \left |\{ \alpha\in A: \ \alpha_{m+1}=s+1 \} \right| \\
    &= \sum_{s=0}^\infty \binom{s}{m} \binom{l-s-1}{k}.
  \end{aligned}
  \]
 Hence, the proof is complete.
\end{proof}

\begin{proof}[Proof of Proposition \ref{prop:reduce}]
  It suffices to show that
  \begin{equation} \label{eq:A012}
    \mathcal A_1^\beta = \mathcal A_0^{\beta+1} + p_\beta \sum_{l=0}^\infty \hat c_l \Delta_l, \quad
    \mathcal A_2^\beta = -\mathcal A_0^\beta + \mathcal A_1^{\beta+1} + p_\beta \sum_{l=1}^\infty \hat c_l \Delta_{l+1}
  \end{equation}
  with
  \begin{equation} \label{eq:zero}
    \sum_{l\ge0} \hat c_l \Delta_l = \sum_{l\ge 1} \hat c_l \Delta_{l+1} = 0
  \end{equation}
  for the moment closure satisfying the assumptions in the proposition. For (\ref{eq:A012}), we first show that
  \begin{equation} \label{eq:pDl01}
    \partial \Delta_l = \Delta_{l+1}, \quad
    \partial^2 \Delta_l = \partial \Delta_{l+1} - \Delta_l.
  \end{equation}
  This follows from the facts $\phi'(\xi)=-\xi\phi(\xi)$, $\phi''(\xi)=(\xi^2-1)\phi(\xi)$ and the integration by parts:
  \[
    \partial \Delta_l = \int_\mathbb R \left( \diffp{}{\xi} \xi^l \right) \phi(\xi) d\xi
    = -\int_\mathbb R \xi^l \phi'(\xi) d\xi = \Delta_{l+1}.
  \]
  The second equality of (\ref{eq:pDl01}) can be verified in a similar manner.
  Then we compute from (\ref{eq:calAdef}) that
  \[
  \begin{aligned}
    \mathcal A_1^\beta &= \sum_{k,l\ge 0} \hat c_{k+l+1} p_{k+\beta} \Delta_{l+1} = \sum_{k\ge -1, \ l\ge 1} \hat c_{k+l+1} p_{k+\beta+1} \Delta_l \\
    &= \mathcal A_0^{\beta+1} + p_\beta\sum_{l\ge1} \hat c_l \Delta_l - \sum_{k\ge 0} \hat c_{k+1} p_{k+\beta+1}
    = \mathcal A_0^{\beta+1} + p_\beta\sum_{l\ge0} \hat c_l \Delta_l,
  \end{aligned}
  \]
  where (\ref{eq:pkpro}) is used in the last step to yield $\sum_{k\ge 0} \hat c_{k+1} p_{k+\beta+1} = - \hat c_0 p_\beta$. The $\mathcal A_2^\beta$-equality in (\ref{eq:A012}) can be derived in a similar manner.

  It remains to verify (\ref{eq:zero}). Since the closure $M_{N+1}(\mathbf M)$ is a homogeneous function of order one, we have
  \[
    M_{N+1} - \sum_{l=0}^N \diffp{{M_{N+1}}}{{M_l}} M_l = 0.
  \]
  Take $\mathbf M=(\Delta_0,\dots,\Delta_N)$ and hence $M_{N+1}=\Delta_{N+1}$ by the assumption. With the definition of $\hat c_l$ in (\ref{eq:cjeq}), we derive $\sum_{l \ge 0} \hat c_l \Delta_l = 0$ immediately.
  The rest part of (\ref{eq:zero}) needs affine invariance of the closure. Set $\sigma=1$ in (\ref{eq:affinv}) and write the $(N+1)$th component as
  \[
    \sum_{j=0}^{N+1}\binom{N+1}{j} M_j u^{N+1-j } = M_{N+1} \circ \mathcal S_{N}^{[u,1]}(\mathbf M).
  \]
  Applying $d/du |_{u=0}$ to the both sides and using the explicit form of $\mathcal S_N^{[u,1]}$ in (\ref{eq:Sk}), we derive
  \[
    (N+1)M_N = \sum_{l=1}^N \diffp{{M_{N+1}}}{{M_l}} \cdot l M_{l-1}.
  \]
  Then the equality $\sum_{l\ge 1}\hat c_l \Delta_{l+1}=0$ follows after assigning $M_k = \Delta_k$ for $k=0,\dots,N+1$ and noticing $l\Delta_{l-1}=\Delta_{l+1}$ (see, for example, Lemma 4.1(a) in \cite{Huang2020}). This justifies (\ref{eq:zero}) and completes the proof.
\end{proof}

\subsection{Proof of Proposition \ref{prop:pkchoice}} \label{subsec:pk}

This subsection is devoted to proving $\mathcal A_0^\beta = 0$ in (\ref{eq:calAdef}) for $\beta=0,\dots,2n-1$ with the $p_k$'s specified in (\ref{eq:pksol}). In what follows the bracket $\langle \cdot \rangle$ is induced by $\mathbf M=(\Delta_m)_{m=0}^{2n}$ on $\mathbb R[X]_{2n}$.

Recall $\hat F=\hat Q_n \hat R_{n+1}$ and write
\[
  \hat Q_n = \sum_{k\ge0}\mu_k X^k, \quad
  \hat R_{n+1} = \sum_{k\ge0}\nu_k X^k
\]
with $\mu_k=0$ for $k<0$ or $k\ge n+1$ and $\nu_k=0$ for $k<0$ or $k\ge n+2$. Then we have
\[
  \hat c_k = \sum_{l=0}^k \mu_l \nu_{k-l}.
\]
Set $\Delta_k=0$ for $k<0$. We compute $\mathcal A_0^\beta$ as follows.

\textbf{Step I}. For $\beta=0$, it is clear from (\ref{eq:calAdef}) that
\[
  \mathcal A_0^0 = \sum_{k\ge0,\ l\ge0} \hat c_{k+l+1} \Delta_k \Delta_l + (n-1)!.
\]
Then a direct calculation gives
\[
\begin{aligned}
  \sum_{k\ge0,\ l\ge0} \hat c_{k+l+1} \Delta_k \Delta_l &= \sum_{m\ge0,\ l\ge0} \hat c_{m+1}\Delta_{m-l}\Delta_l \\
  &= \sum_{m\ge0,\ l\ge0} \left( \sum_{0\le k\le l} + \sum_{k\ge l+1} \right) \mu_k \nu_{m+1-k}\Delta_{m-l}\Delta_l:= \text{I}+\text{II}.
\end{aligned}
\]
For the first part, we have
\[
\begin{aligned}
  \text{I}&=\sum_{m\ge0,\ k\ge0} \mu_k \nu_{m+1-k} \sum_{l\ge k} \Delta_{m-l}\Delta_l
  =\sum_{m\ge0,\ k\ge0} \mu_{m+1-k} \nu_{k} \sum_{l=0}^{k-1}\Delta_l \Delta_{m-l} \\
  &= \sum_{k\ge0}\nu_k\sum_{l=0}^{k-1}\Delta_l\sum_{m\ge0}\mu_{m+1-k}\Delta_{m-l}
  = \sum_{k\ge 0} \nu_k \sum_{l=0}^{k-1} \Delta_l \langle X^{k-l-1} \hat Q_n \rangle = n!.
\end{aligned}
\]
Here the second equality is derived after a change of variables $k\leftarrow m+1-k$ and $l\leftarrow m-l$. The last equality holds as the summand is nonzero only when $k=n+1$ and $l=0$ thanks to the orthogonality.

Similarly, we have
\[
\begin{aligned}
  \text{II} &= \sum_{l\ge0}\sum_{k\ge l+1} \mu_k\Delta_l \sum_{m\ge0}\nu_{m+1-k}\Delta_{m-l} \\
  &= \sum_{l\ge0}\sum_{k\ge l+1} \mu_k\Delta_l \langle X^{k-l-1} \hat R_{n+1}\rangle = -(n+1)(n-1)!.
\end{aligned}
\]
The last equality holds because $\hat R_{n+1}=\hat Q_{n+1} - (n+1)\hat Q_{n-1}$ and thus the summand is nonzero only when $k=n$ and $l=0$, thanks again to the orthogonality. Therefore, $\mathcal A_0^0=0$ follows immediately.

\textbf{Step II}. For $\beta =1, \dots ,2n-1$, $\mathcal A_0^\beta$ in (\ref{eq:calAdef}) contains $p_{2n+1},\dots,p_{2n+\beta}$ which need to be handled. For this purpose, we use a variant of (\ref{eq:pkpro}):
\[
  \sum_{k\ge-l-1} \hat c_{k+l+1} p_k = 0 \quad \text{for  } l\le -1
\]
to convert $\mathcal A_0^\beta$ in (\ref{eq:calAdef}) to
\begin{equation} \label{eq:rt1}
\begin{split}
  \mathcal A_0^\beta &= \sum_{k\ge\beta}\sum_{l\ge-\beta} p_k \hat c_{k+l+1} \Delta_{l+\beta}
  = \left( \sum_{\substack{k\ge\beta\\ l\ge0}} + \sum_{\substack{k\ge\beta \\ -\beta\le l\le -1}} \right)p_k \hat c_{k+l+1} \Delta_{l+\beta} \\
  &= \left( \sum_{\substack{k\ge\beta\\ l\ge0}} - \sum_{\substack{-l-1\le k\le\beta-1 \\ -\beta\le l\le -1}} \right)p_k \hat c_{k+l+1} \Delta_{l+\beta}
  :=A_0^\beta - B_0^\beta.
\end{split}
\end{equation}
Now $A_0^\beta$ and $B_0^\beta$ only include nonzero summands containing $p_0,\dots,p_{2n}$.

With $p_k$ given in (\ref{eq:pksol}), we compute
\begin{equation} \label{eq:rt2}
  \begin{split}
      A_0^\beta &- (n-1)!\Delta_\beta = \sum_{k\ge\beta,\ l\ge0} \hat c_{k+l+1} \Delta_k \Delta_{l+\beta} = \sum_{m\ge\beta} \sum_{0\le l\le m-\beta} \hat c_{m+1}\Delta_{m-l}\Delta_{l+\beta} \\
    &=  \sum_{m\ge\beta} \sum_{0\le l\le m-\beta} \left( \sum_{0\le k\le l} + \sum_{k\ge l+1} \right) \mu_k \nu_{m+1-k}\Delta_{m-l}\Delta_{l+\beta}:= \text{I}_A+\text{II}_A.
  \end{split}
\end{equation}
Similar to $\beta=0$ in \textbf{Step I}, we obtain
\begin{equation} \label{eq:rt3}
\begin{split}
  \text{I}_A&=\sum_{m\ge\beta}\sum_{k=0}^{m-\beta} \mu_k \nu_{m+1-k} \sum_{l=k}^{m-\beta} \Delta_{m-l}\Delta_{l+\beta}
  =\sum_{m\ge0}\sum_{k=\beta+1}^{m+1} \mu_{m+1-k} \nu_k \sum_{l=0}^{k-1-\beta}\Delta_{l+\beta} \Delta_{m-l} \\
  &= \sum_{k\ge0}\nu_k\sum_{l=0}^{k-1-\beta}\Delta_{l+\beta} \sum_{m\ge k-1}\mu_{m+1-k}\Delta_{m-l} \\
  &= \sum_{k\ge 0} \nu_k \sum_{l=0}^{k-1-\beta} \Delta_{l+\beta} \langle X^{k-l-1} \hat Q_n \rangle =
  \left \{
    \begin{aligned}
      &n!\Delta_\beta, \quad &&\text{for} \quad \beta \le n, \\
      &0,              \quad &&\text{for} \quad \beta \ge n+1.
    \end{aligned}
  \right.
\end{split}
\end{equation}
Here the second equality is derived after a change of variables $k\leftarrow m+1-k$ and $l\leftarrow m-\beta-l$. The last equality holds as the summand is nonzero only when $\beta \le n$, $k=n+1$ and $l=0$ thanks to the orthogonality.
Moreover, it is seen that
\begin{equation} \label{eq:rt4}
\begin{split}
  \text{II}_A &= \sum_{l\ge0} \Delta_{l+\beta} \sum_{k\ge l+1}\mu_k \sum_{m\ge l+\beta} \nu_{m+1-k}\Delta_{m-l} \\
  &= \sum_{l\ge0} \Delta_{l+\beta} \sum_{k\ge l+1}\mu_k \left( \langle X^{k-l-1}\hat R_{n+1} \rangle - \sum_{m=k-1}^{l+\beta-1}\nu_{m+1-k}\Delta_{m-l} \right)\\
  &= -(n+1)(n-1)!\Delta_\beta - \tilde A
\end{split}
\end{equation}
with
\[
\begin{aligned}
  \tilde A :&= \sum_{l\ge0} \sum_{k\ge l+1} \sum_{m=k-1}^{l+\beta-1} \mu_k \nu_{m+1-k} \Delta_{m-l} \Delta_{l+\beta} \\
  &= \sum_{m\ge0} \sum_{l=\max(0,m+1-\beta)}^m \sum_{k=l+1}^{m+1} \mu_k \nu_{m+1-k} \Delta_{m-l} \Delta_{l+\beta} \\
  &= \sum_{m\ge 0} \sum_{l=-\beta}^{\min(m-\beta,-1)} \sum_{k=0}^{l+\beta} \mu_{m+1-k} \nu_k \Delta_{l+\beta} \Delta_{m-l} \\
  &= \sum_{k=0}^{\beta-1} \nu_k \sum_{l=k-\beta}^{-1} \Delta_{l+\beta} \sum_{m=l+\beta}^\infty \mu_{m+1-k}\Delta_{m-l}.
\end{aligned}
\]
Here we perform a change of variables $k\leftarrow m+1-k$ and $l\leftarrow m-\beta-l$ in the third equality.

Next we treat
\[
\begin{aligned}
  B_0^\beta &= \sum_{l=-\beta}^{-1}\sum_{k=-1-l}^{\beta-1}\hat c_{k+l+1}\Delta_k\Delta_{l+\beta}
  =\sum_{m=-1}^{\beta-2}\sum_{l=m+1-\beta}^{-1}\hat c_{m+1} \Delta_{m-l}\Delta_{l+\beta} \\
  &= \sum_{m=-1}^{\beta-2}\sum_{l=m+1-\beta}^{-1} \sum_{k=0}^{m+1} \mu_{m+1-k}\nu_k \Delta_{m-l}\Delta_{l+\beta} \\
  &= \sum_{k=0}^{\beta-1} \nu_k \sum_{l=k-\beta}^{-1} \Delta_{l+\beta} \sum_{m=k-1}^{l+\beta-1} \mu_{m+1-k} \Delta_{m-l},
\end{aligned}
\]
which implies that
\begin{equation} \label{eq:rt5}
  \tilde A + B_0^\beta = \sum_{k=0}^{\beta-1} \nu_k \sum_{l=k-\beta}^{-1} \Delta_{l+\beta} \langle X^{k-l-1} \hat Q_n \rangle =
  \left \{
    \begin{aligned}
      0, \quad &\text{for} \quad \beta \le n, \\
      -n!\Delta_\beta, \quad &\text{for} \quad \beta \ge n+1.
    \end{aligned}
  \right.
\end{equation}
Indeed, for $\beta \le n$, all the summands are zero as $k-l-1<\beta\le n$ holds. Then, for $n+1\le \beta \le 2n-1$, it is possible to take $l=k-1-n \in [k-\beta,-1]$ if $k\le n\le \beta-1$. This results in
\[
  \tilde A + B_0^\beta = n!\sum_{k=0}^n \nu_k \Delta_{k-1-n+\beta} = n! \left(\langle X^{\beta-n-1} \hat R_{n+1}\rangle - \Delta_\beta \right) = -n!\Delta_\beta,
\]
since $\beta-n-1\le n-2$.

Now we collect the above results in (\ref{eq:rt1})-(\ref{eq:rt5}) to see that
\[
  \mathcal A_0^\beta = (n-1)!\Delta_\beta + \text{I}_A - (n+1)(n-1)!\Delta_\beta - (\tilde A + B_0^\beta) = 0
\]
for both $\beta \le n$ and $\beta \ge n+1$. This proves Proposition \ref{prop:pkchoice}.

\section{Conclusions} \label{sec:concld}

This paper presents a purely analytic proof of the hyperbolicity of a quadrature method of moments (called HyQMOM) for the one-dimensional BGK equation. The method takes advantage of the orthogonal polynomials associated with realizable moments and infers unclosed terms by constructing higher-order orthogonal polynomials.
Our proof is based on a factorization of the characteristic polynomial for the resultant first-order PDE and a polynomial-induced closure technique.
As a byproduct, a class of numerical schemes for the HyQMOM system is shown to be realizability preserving under CFL-type conditions.

Furthermore, we show that the system preserves the dissipative property of the kinetic equation by verifying the structural stability condition in section \ref{subsec:ssc}. The proof involves seeking positive solutions to an overdetermined system of algebraic equations. It uses the newly introduced affine invariance and homogeneity of the HyQMOM and heavily relies on the theory of orthogonal polynomials associated with realizable moments, in particular, the moments of the standard normal distribution.

Finally, the developed ideas can be used to analyze other quadrature-based moment methods, including different variants of the EQMOM.

\section*{Acknowledgments}
The authors are grateful to Prof. Ruo Li at Peking University and Prof. Kailiang Wu at SUSTECH for insightful discussions.

\section*{Funding sources}
This work was supported by the National Key Research and Development Program of China (Grant no. 2021YFA0719200) and the National Natural Science Foundation of China (Grant no. 12071246).

\section*{Declarations of interest}
None.

\appendix

\section{Appendix}

This appendx is devoted to the proofs of Proposition \ref{prop:afinvold} and Proposition \ref{prop:qeqhomo}.

\begin{proof} [Proof of Proposition \ref{prop:afinvold}]
  \textbf{(i)}. As in section \ref{subsec:hyafinv}, we can show that $\bar{\mathbf M}'=(\bar M_k)_{k=0}^{2n-1}:=\mathcal S_{2n-1}^{[u,\sigma]}(\mathbf M')$ is expressed as $\bar M_k = \sum_{i=1}^{n} w_i (\sigma u_i + u)^k$ for $k=0,...,2n-1$.
  Given $\bar{\mathbf M}'$, the QMOM closure reads as
  \[
    M_{2n}(\bar{\mathbf M}') = \sum_{i=1}^n w_i (\sigma u_i + u)^{2n} = \sum_{j=0}^{2n} \binom{2n}{j} \sigma^j M_j u^{2n-j}=\bar{M}_{2n},
  \]
where $\bar{M}_{2n}$ is just the last component of $\mathcal S_{2n}^{[u,\sigma]}((\mathbf M',M_{2n}))$. This proves the affine invariance of the QMOM.

  \textbf{(ii)}. For the EQMOM, $\mathbf M\in\mathbb R^{2n+1}$ is realized by the reconstructed distribution
  \[
    f(\xi;W) = \sum_{i=1}^n \frac{w_i}{\sigma} \calk \left( \frac{\xi-u_i}{\sigma} \right)
  \]
  with $W=(w_i,u_i,\sigma)_{i=1}^n$ ($w_i>0$ and $\sigma>0$) uniquely determined by $\mathbf M$.
  By the definition of the EQMOM closure and (\ref{eq:scalef}), it is obvious that the moments $\mathcal S_{2n+1}^{[u,\sigma']} \circ \mathcal M(\mathbf M)$ are realized by the shifted and rescaled distribution
  \[
    \frac{1}{\sigma'} f\left( \frac{\xi-u}{\sigma'}; W \right) = f(\xi;W')
  \]
  with $W'=(w_i,\sigma'u_i+u,\sigma\sigma')_{i=1}^n$ for any $u\in\mathbb R$ and $\sigma'>0$.
  On the other hand, the same argument implies that $\mathcal M \circ \mathcal S_{2n}^{[u,\sigma']}(\mathbf M)$ is realized by $f(\xi;W')$.
  Hence they are equal and the proof is completed.
\end{proof}

\begin{proof}[Proof of Proposition \ref{prop:qeqhomo}]
  \textbf{(i)}. For the QMOM, we know from Theorem 2.3 of \cite{Huang2020} that the characteristic polynomial is $F(X;\mathbf M')=Q_n^2(X;\mathbf M')$. Taking $G(X;\mathbf M')=F(X;\mathbf M')$, we immediately see from $\langle Q_n^2\rangle_{\mathbf M}=0$ that the QMOM (\ref{eq:qm2}) is of the form (\ref{eq:clos}) with $N=2n-1$.

  \textbf{(ii)}. For the EQMOM with kernel (\ref{eq:eqker}), we assume that all moments of the univariate function $\calk(\xi)$ are finite. Moerover, we refer to Remark 2.1 of \cite{ZHY2023} and may as well assume $\frkm_0=\frkm_2=1$ and $\frkm_1=0$.
  Under these assumptions, it is shown in \cite{ZHY2023} (Eq.(4.3) and Proposition 4.1 therein) that the characteristic polynomial $F(X;\mathbf M)=\mathcal D_\sigma g$, where $g\in\mathbb R[X]_{2n+1}$ and the operator $\mathcal D_\sigma$ are
  \[
    g = (X-\xi_1)^2\cdots (X-\xi_n)^2 (X-\tilde u)\quad \text{and} \quad\mathcal D_\sigma = \sum_{k=0}^\infty \mathfrak b_k \sigma^k \partial^k
  \]
  with $\mathfrak b_k$ constructed iteratively as $\mathfrak b_0=1$ and
  \begin{equation} \label{eq:bdef}
    \mathfrak b_k = -\sum_{j=1}^k \frac{\frkm_j}{j!}\mathfrak b_{k-j}
  \end{equation}
  for $k\ge 1$. Explicit forms of $\tilde u=\tilde u(\mathbf M)$ can be found in \cite{ZHY2023} (for general $\calk(\xi)$ and $n=2$) and \cite{Huang2020} (for the Gaussian kernel $\calk(\xi)=\phi(\xi)$ and all $n\ge 1$).

  Taking $G(X;\mathbf M)=F(X;\mathbf M)$, we will show that the EQMOM closure is just $M_{2n+1}=\langle X^{2n+1}-G \rangle_\mathbf M$, which is of the form (\ref{eq:clos}) with $N=2n$.
For this purpose, we compute
\[
  \begin{aligned}
    \int_{\mathbb R} \frac{1}{\sigma} \calk \left( \frac{\xi-X}{\sigma} \right) \mathcal D_\sigma \xi^j d\xi
    &= \sum_{k=0}^j \frac{j!}{(j-k)!}\mathfrak b_k \sigma^k \int_{\mathbb R} \frac{1}{\sigma} \calk \left( \frac{\xi-X}{\sigma} \right) \xi^{j-k} d\xi \\
    &= \sum_{k=0}^j \frac{j!}{(j-k)!}\mathfrak b_k \sigma^k \int_{\mathbb R} \calk(\xi) (X+\sigma \xi)^{j-k} d\xi \\
    &= \sum_{k=0}^j \sum_{l=0}^{j-k} \frac{j!}{l!(j-k-l)!} \mathfrak b_k \sigma^{k+l} X^{j-k-l} \frkm_l \\
    &= \sum_{s=0}^j \frac{j!}{(j-s)!} \sigma^s X^{j-s} \left( \sum_{l=0}^s \frac{\frkm_l}{l!} \mathfrak b_{s-l} \right)=X^j,
  \end{aligned}
  \]
where the last step follows from the recursive relation (\ref{eq:bdef}). This immediately gives the identity
  \[
    P(X) = \int_{\mathbb R} \frac{1}{\sigma} \calk \left( \frac{\xi-X}{\sigma} \right) \mathcal D_\sigma P(\xi) d\xi
  \]
  for any polynomial $P(X)$.

With this identity, for the EQMOM ansatz
  \[
    f(\xi)=\sum_{i=1}^n \frac{w_i}{\sigma} \calk \left( \frac{\xi-\xi_i}{\sigma} \right)
  \]
  we deduce that
  \[
    \int_{\mathbb R} f(\xi) G(\xi;\mathbf M) d\xi = \sum_{i=1}^n \frac{w_i}{\sigma} \int_{\mathbb R} \calk \left( \frac{\xi-\xi_i}{\sigma} \right) \mathcal D_\sigma g(\xi) d\xi
    =  \sum_{i=1}^n w_i g(\xi_i) = 0
  \]
and thereby
\[
M_{2n+1} = \int_\mathbb R \xi^{2n+1} f d\xi=\int_\mathbb R (\xi^{2n+1}-G(\xi;M)) f d\xi=\langle X^{2n+1}-G \rangle_\mathbf M.
\]
This is of the form (\ref{eq:clos}) and hence the proof is complete.
\end{proof}

\bibliographystyle{plain}
\bibliography{references}

\begin{thebibliography}{10}

\bibitem{bgk1954}
P.~L. Bhatnagar, E.~P. Gross, and M.~Krook.
\newblock A model for collision processes in gases. {I}. small amplitude
  processes in charged and neutral one-component systems.
\newblock {\em Phys. Rev.}, 94(3):511--525, 1954.

\bibitem{Boh2020}
Niclas Böhmer and Manuel Torrilhon.
\newblock Entropic quadrature for moment approximations of the boltzmann-bgk
  equation.
\newblock {\em J. Comput. Phys.}, 401:108992, 2020.

\bibitem{Cercignani1988}
Carlo Cercignani.
\newblock {\em The Boltzmann Equation and Its Applications}.
\newblock Springer, New York, 1988.

\bibitem{Chalons2010}
C.~Chalons, R.~O. Fox, and M.~Massot.
\newblock A multi-{G}aussian quadrature method of moments for gas-particle
  flows in a {LES} framework.
\newblock In {\em Proceedings of the Summer Program 2010}, pages 347--358.
  Center for Turbulence Research, Stanford University, Stanford, CA, 2010.

\bibitem{Chalons2017}
C.~Chalons, R.O. Fox, F.~Laurent, M.~Massot, and A.~Vi\'{e}.
\newblock Multivariate {G}aussian extended quadrature method of moments for
  turbulent disperse multiphase flow.
\newblock {\em Multiscale Model. Simul.}, 15(4):1553--1583, 2017.

\bibitem{Chalons2012}
Christophe Chalons, Damieh Kah, and Marc Massot.
\newblock Beyond pressureless gas dynamics: Quadrature-based velocity moment
  models.
\newblock {\em Commun. Math. Sci.}, 10(4):1241--1272, 2012.

\bibitem{ClementsMars2013}
J.~Sidney Clements, Samuel~M. Thompson, Nathanael~D. Cox, Michael~R. Johansen,
  Blakeley~S. Williams, Michael~D. Hogue, M.~Loraine Lowder, and Carlos~I.
  Calle.
\newblock Development of an electrostatic precipitator to remove martian
  atmospheric dust from isru gas intakes during planetary exploration missions.
\newblock {\em IEEE Trans. Ind. Appl.}, 49(6):2388--2396, 2013.

\bibitem{Di2017}
Yana Di, Yuwei Fan, Ruo Li, and Lingchao Zheng.
\newblock Linear stability of hyperbolic moment models for {B}oltzmann
  equation.
\newblock {\em Numer. Math. Theor. Meth. Appl.}, 10(2):255--277, May 2017.

\bibitem{Fox2008}
Rodney~O. Fox.
\newblock A quadrature-based third-order moment method for dilute gas-particle
  flows.
\newblock {\em J. Comput. Phys.}, 227:6313--6350, 2008.

\bibitem{FoxLau2022}
Rodney~O. Fox and Fr\'{e}d\'{e}rique Laurent.
\newblock Hyperbolic quadrature method of moments for the one-dimensional
  kinetic equation.
\newblock {\em SIAM J. Appl. Math.}, 82(2):750--771, 2022.

\bibitem{Fox2018}
Rodney~O. Fox, Frédérique Laurent, and Aymeric Vié.
\newblock Conditional hyperbolic quadrature method of moments for kinetic
  equations.
\newblock {\em J. Comput. Phys.}, 365:269--293, 2018.

\bibitem{Gau2004}
Walter Gautschi.
\newblock {\em {Orthogonal Polynomials: Computation and Approximation}}.
\newblock Oxford University Press, 04 2004.

\bibitem{Grad1949}
H.~Grad.
\newblock On the kinetic theory of rarefield gases.
\newblock {\em Comm. Pure Appl. Math.}, 2:331--407, 1949.

\bibitem{Huang2022}
Qian Huang, Julian Koellermeier, and Wen-An Yong.
\newblock Equilibrium stability analysis of hyperbolic shallow water moment
  equations.
\newblock {\em Math. Meth. Appl. Sci.}, 45(10):6459--6480, 2022.

\bibitem{Huang2020}
Qian Huang, Shuiqing Li, and Wen-An Yong.
\newblock Stability analysis of {Q}uadrature-{B}ased {M}oment {M}ethods for
  kinetic equations.
\newblock {\em SIAM J. Appl. Math.}, 80(1):206--231, 2020.

\bibitem{HuangEF2}
Qian Huang, Peng Ma, Linchao Cai, and Shuiqing Li.
\newblock Kinetic simulation of fine particulate matter evolution and
  deposition in a 25 kw pulverized coal combustor.
\newblock {\em ENERGY FUELS}, 34(12):15389--15398, DEC 17 2020.

\bibitem{LiMar2022}
Dongyue Li and Daniele Marchisio.
\newblock Implementation of chyqmom in openfoam for the simulation of
  non-equilibrium gas-particle flows under one-way and two-way coupling.
\newblock {\em Powder Technol.}, 396(B):765--784, JAN 2022.

\bibitem{LZH2019}
Zhi-Hui Li, Ao-Ping Peng, Qiang Ma, Lei-Ning Dang, Xiao-Wei Tang, and Xue-Zhou
  Sun.
\newblock Gas-kinetic unified algorithm for computable modeling of boltzmann
  equation and application to aerothermodynamics for falling disintegration of
  uncontrolled tiangong-no.1 spacecraft.
\newblock {\em Adv. Aerodyn.}, 1(1), FEB 22 2019.

\bibitem{LiuJW2016}
Jiawei Liu and Wen-An Yong.
\newblock Stability analysis of the {B}iot/{S}quirt models for wave propagation
  in saturated porous media.
\newblock {\em Geophys. J. Int.}, 204(1):535--543, January 2016.

\bibitem{March2013}
M.~C. Marchetti, J.~F. Joanny, S.~Ramaswamy, T.~B. Liverpool, J.~Prost, Madan
  Rao, and R.~Aditi Simha.
\newblock Hydrodynamics of soft active matter.
\newblock {\em Rev. Mod. Phys.}, 85(3), JUL 19 2013.

\bibitem{MarFox2013}
Daniele~L. Marchisio and Rodney~O. Fox.
\newblock {\em Computational Models for Polydisperse Particulate and Multiphase
  Systems}.
\newblock Cambridge University Press, Cambridge, 2013.

\bibitem{Mc1997}
R.~McGraw.
\newblock Description of aerosol dynamics by the quadrature method of moments.
\newblock {\em Aerosol Sci. Technol.}, 27:255--265, 1997.

\bibitem{Morris2011}
A.~B. Morris, D.~B. Goldstein, P.~L. Varghese, and L.~M. Trafton.
\newblock {Plume Impingement on a Dusty Lunar Surface}.
\newblock {\em AIP Conf. Proc.}, 1333(1):1187--1192, 05 2011.

\bibitem{gtm277}
Konrad Schmüdgen.
\newblock {\em The Moment Problem. Graduate Texts in Mathematics, vol 277}.
\newblock Springer, Cham, 2017.

\bibitem{Serre1999}
D.~Serre.
\newblock {\em Systems of Conservation Laws 1: Hyperbolicity, Entropies, Shock
  Waves}.
\newblock Cambridge University Press, Cambridge, 1999.

\bibitem{Mueller2023}
Pierre-Yves~C.R. Taunay and Michael~E. Mueller.
\newblock Quadrature-based moment methods for kinetic plasma simulations.
\newblock {\em J. Comput. Phys.}, 473:111700, 2023.

\bibitem{van2021}
Maxim {Van Cappellen}, Maria~Rosaria Vetrano, and Delphine Laboureur.
\newblock Higher order hyperbolic quadrature method of moments for solving
  kinetic equations.
\newblock {\em J. Comput. Phys.}, 436:110280, 2021.

\bibitem{Yong1999}
Wen-An Yong.
\newblock Singular perturbations of fisrt-order hyperbolic systems with stiff
  source terms.
\newblock {\em J. Differ. Equations}, 155(1):89--132, June 1999.

\bibitem{Yong2001}
Wen-An Yong.
\newblock {\em Basic Aspects of Hyperbolic Relaxation Systems}, pages 259--305.
\newblock Birkh{\"a}user Boston, Boston, MA, 2001.

\bibitem{Yong2008}
Wen-An Yong.
\newblock {An interesting class of partial differential equations}.
\newblock {\em J. Math. Phys.}, 49(3):033503, 03 2008.

\bibitem{ZHY2023}
Ruixi Zhang, Qian Huang, and Wen-An Yong.
\newblock Stability analysis of an extended quadrature method of moments for
  kinetic equations.
\newblock {\em SIAM J. Math. Anal.}, 2024.

\bibitem{Zhao2017}
W.~Zhao, Wen-An Yong, and Li-Shi Luo.
\newblock Stability analysis of a class of globally hyperbolic moment systems.
\newblock {\em Commun. Math. Sci.}, 15:609--633, 2017.

\end{thebibliography}

\end{document}